\documentclass[a4paper,11pt]{article}

\usepackage[T1]{fontenc}
\usepackage{lmodern}

\usepackage{dsfont}

\usepackage[latin1]{inputenc}
\usepackage{amsmath}
\usepackage{amsthm}
\usepackage{amssymb}
\usepackage{mathrsfs}
\usepackage{graphicx}
\usepackage[all]{xy}
\usepackage{hyperref}

\usepackage{xcolor}

\usepackage{makeidx}

\usepackage{stmaryrd}
\usepackage{caption}

\usepackage{abstract} 

\newtheorem{thm}{Theorem}[section]
\newtheorem{cor}[thm]{Corollary}
\newtheorem{claim}[thm]{Claim}
\newtheorem{fact}[thm]{Fact}

\newtheorem{lemma}[thm]{Lemma}
\newtheorem{prop}[thm]{Proposition}
\newtheorem{conj}[thm]{Conjecture}

\theoremstyle{definition}
\newtheorem{definition}[thm]{Definition}

\newtheorem{remark}[thm]{Remark}

\def\rquotient#1#2{%
	\makeatletter
	\raise.3ex\hbox{$#1$}/\lower.3ex\hbox{$#2$}%
	\makeatother
}	

\makeatletter
\newcommand{\subjclass}[2][2010]{%
	\let\@oldtitle\@title%
	\gdef\@title{\@oldtitle\footnotetext{#1 \emph{Mathematics subject classification.} #2}}%
}
\newcommand{\keywords}[1]{%
	\let\@@oldtitle\@title%
	\gdef\@title{\@@oldtitle\footnotetext{\emph{Key words and phrases.} #1.}}%
}
\makeatother

\newcommand{\Address}{{% additional braces for segregating \footnotesize
		\bigskip
		\small
		
		\textsc{D\'epartement de Math\'ematiques B\^atiment 307, Facult\'e des Sciences d'Orsay, Universit\'e Paris-Sud, F-91405 Orsay Cedex, France.}\par\nopagebreak
		\textit{E-mail address}: \texttt{anthony.genevois@math.u-psud.fr}
		
}}

\makeindex

\title{Cyclic hyperbolicity in CAT(0) cube complexes}
\date{\today}
\author{Anthony Genevois}


\begin{document}

\maketitle

\begin{abstract}
It is known that a cocompact special group $G$ does not contain $\mathbb{Z} \times \mathbb{Z}$ if and only if it is hyperbolic; and it does not contain $\mathbb{F}_2 \times \mathbb{Z}$ if and only if it is toric relatively hyperbolic. Pursuing in this direction, we show that $G$ does not contain $\mathbb{F}_2 \times \mathbb{F}_2$ if and only if it is weakly hyperbolic relative to cyclic subgroups, or cyclically hyperbolic for short. This observation motivates the study of cyclically hyperbolic groups, which we initiate in the class of groups acting geometrically on CAT(0) cube complexes. Given a cubulable cyclically hyperbolic group $G$, we first prove a structure theorem: $G$ virtually splits as the direct sum of a free abelian group and an acylindrically hyperbolic cubulable group. Next, we prove a strong Tits alternative: every subgroup $H \leq G$ either is virtually abelian or or contains a finite-index subgroup whose commutator subgroup is acylindrically hyperbolic. As a consequence, $G$ is SQ-universal and it cannot contain subgroups such as products of free groups and virtually simple groups. 
\end{abstract}

\tableofcontents

\section{Introduction}

\noindent
A major theme in geometric group theory is to find algebraic characterisations of various (quasi-)geometric properties of finitely generated groups. In this direction, it has been observed that being (Gromov-)hyperbolic often amounts to not containing $\mathbb{Z}^2$ as a subgroup. Of course, this is not true in full generality, as shown for instance by Baumslag-Solitar groups and lamplighter groups. But it turns out to be true for many families of groups exhibiting a nonpositively curved behaviour, even though it is still unknown whether this holds for arbitrary CAT(0) groups or even for cubulable groups.

\medskip \noindent
A large class of CAT(0) groups for which it is known that being hyperbolic is equivalent to not containing $\mathbb{Z}^2$ as a subgroup is given by (virtually) \emph{cocompact special groups}, as introduced in \cite{MR2377497}. This is a theorem proved by Caprace and Haglund in \cite{SpecialHyp}. In \cite{SpecialRH}, we introduced a new combinatorial formalism that, roughly speaking, allows us to study special groups similarly to right-angled Artin groups, highlighting a strong connection between product subcomplexes and product subgroups. As an application, we generalised \cite{SpecialHyp} by proving that a (virtually) cocompact special group is toric relatively hyperbolic if and only if it does not contain $\mathbb{F}_2 \times \mathbb{Z}$ as a subgroup.

\medskip \noindent
In view of these results, a natural question rises: is there a property of negative curvature corresponding to not containing $\mathbb{F}_2 \times \mathbb{F}_2$ as a subgroup? The first main result of this article answers this question as follows:

\begin{thm}\label{thm:Special}
A virtually cocompact special group is cyclically hyperbolic if and only if it does not contain $\mathbb{F}_2 \times \mathbb{F}_2$ as a subgroup.
\end{thm}

\noindent
Here, \emph{cyclically hyperbolic} is a shorthand for \emph{weakly hyperbolic relative to a finite family of cyclic subgroups}. Recall from \cite{MR1650094} that a finitely generated group $G$ is weakly hyperbolic relative to some collection of subgroups $\mathcal{H}$ if the cone-off of a Cayley graph of $G$ (constructed from an arbitrary finite generating set) over the cosets of subgroups in $\mathcal{H}$ is (Gromov-)hyperbolic. 

\medskip \noindent
It is known that there is no global theory of weakly relatively hyperbolic groups similar to the theory of (strongly) relatively hyperbolic groups. For instance, it is shown in \cite{MR1929719} that there exist finitely generated groups that are weakly hyperbolic relative to hyperbolic groups but that are not finitely presented, that have unsolvable word problem, or that are simple. However, these constructions do not provide cyclically hyperbolic groups. Theorem~\ref{thm:Special} suggests that an investigation of cyclically hyperbolic groups could be of interest. Observe, however, that containing a product of free groups is not an obstruction to cyclic hyperbolicity in full generality. Indeed, obvious examples of cyclically hyperbolic groups are boundedly generated groups. These include $\mathrm{SL}(n,\mathbb{Z})$ for $n \geq 3$, which contains products of non-abelian groups (at least for $n \geq 4$). Even worse, there exists a boundedly generated (finitely presented) group containing an isomorphic copy of every recursively presented group \cite[Section~6]{MR2181790}. Nevertheless, it is reasonable to think that cyclically hyperbolic groups satisfy hyperbolic properties relative to their subsets living in products of cyclic subgroups (where products here refer to the group law) and that these subsets are ``small'' in many classes of nonpositively-curved groups met in geometric group theory. 

\medskip \noindent
Of course, cyclic hyperbolicity might not be the property to look at in full generality. There might exist a property that remains to define and that turns out to be equivalent to cyclic hyperbolicity among cocompact special groups. The main objective of this article is to motivate the idea that cyclic hyperbolicity is of interest at least in the family of groups acting geometrically on CAT(0) cube complexes. We emphasize that (cubulable) cyclically hyperbolic groups do not contain neither are contained in acylindrically or relatively hyperbolic groups. For instance, the free product $\mathbb{Z} \ast (\mathbb{F}_2 \times \mathbb{F}_2)$ is relatively hyperbolic but not cyclically hyperbolic (as a consequence of Corollary~\ref{cor:F2xF2}) and the direct product $\mathbb{Z} \times \mathbb{F}_2$ is cyclically hyperbolic but not acylindrically hyperbolic. Right-angled Artin groups $A(\Gamma)$ given by connected square-free graphs $\Gamma$ also provide many examples of cyclically hyperbolic groups (see Proposition~\ref{prop:RAAGcyc}) that are not relatively hyperbolic (see e.g.\ \cite{MR2364824}). Also, it is worth mentioning that it follows from \cite{MR3217625} that cubulable groups are always weakly hyperbolic relative their hyperplane-stabilisers. However, the structure of groups acting geometrically on CAT(0) cube complexes with cyclic (or more generally, virtually abelian) hyperplane-stabilisers is very restrictive; see for instance \cite{MR4579353}. Therefore, typically one cannot deduce some cyclic hyperbolicity from \cite{MR3217625}. And, as mentioned earlier, more general weak relative hyperbolicity usually does not bring valuable information. For instance, uniform lattices in products of trees are weakly hyperbolic relative to virtually free groups, but they may not be virtually torsion-free \cite{MR2694733} or they may be simple \cite{BurgerMozes}. Therefore, \cite{MR3217625} does not bring the kind of information we are interested in here.

\medskip \noindent
The typical example of a cubulable group that is not cyclically hyperbolic is $\mathbb{F}_2 \times \mathbb{F}_2$. More generally, if a group $G$ acts geometrically on a CAT(0) cube complex $X$, then the cyclic hyperbolicity of $G$ prevents $X$ from containing a subspace that ``looks like'' a product of two bushy trees. For instance, the cube complex cannot have two non-Euclidean factors. Combining this fact in particular with the Rank Rigidity Theorem \cite{MR2827012}, this implies that implies that:

\begin{thm}[Structure theorem]\label{thm:Structure}
Let $G$ be a group acting geometrically on a CAT(0) cube complex. If $G$ is cyclically hyperbolic, then it virtually splits as a direct sum $A \times H$ where $A$ is a free abelian group and where $H$ is either trivial or an acylindrically hyperbolic group that acts geometrically on a CAT(0) cube complex.
\end{thm}

\noindent
Because acylindrically hyperbolic groups are SQ-universal \cite{DGO}, it immediately follows that:

\begin{cor}\label{cor:SQuniversal}
Let $G$ be a group acting geometrically on a CAT(0) cube complex. If $G$ is cyclically hyperbolic, then it is either virtually abelian or SQ-universal.
\end{cor}

\noindent
Recall that a group $G$ is \emph{SQ-universal} if every countable subgroup is isomorphic to a subgroup of some quotient of $G$. This implies that $G$ contains uncountably many normal subgroups, so SQ-universality can be thought of as a strong negation of being simple. Consequently, cyclic hyperbolicity prevents the exotic behaviours exhibited in BMW-groups \cite{MR3931408} such as the examples constructed in \cite{MR2694733, BurgerMozes}. 

\medskip \noindent
This idea of forbidding subcomplexes that look like products of bushy trees also leads to interesting consequences on the structure of subgroups. Our conclusion is weaker than Theorem~\ref{thm:Structure}, but it still provides a strong version of the Tits alternative.

\begin{thm}[Strong Tits alternative]\label{thm:TitsAlt}
 Let $G$ be a group acting geometrically on a CAT(0) cube complex. If $G$ is cyclically hyperbolic, then every subgroup $H \leq G$ either is virtually free abelian or contains a finite-index subgroup whose commutator subgroup is acylindrically hyperbolic.  
%Let $G$ be a group acting geometrically on a CAT(0) cube complex. If $G$ is cyclically hyperbolic, then every subgroup $H \leq G$ either is virtually free abelian or admits a sequence of subgroups $H=H_0 \geq H_1 \geq \cdots \geq H_k$ (with $k \leq \dim(X)+1$) such that $H_k$ is acylindrically hyperbolic and, for every $0 \leq i \leq k-1$, $H_{i+1}$ is normal in $H_i$ with $H_i/H_{i+1}$ either finite or free abelian of finite rank. 
\end{thm}

\noindent
For instance, Theorem~\ref{thm:TitsAlt} shows that our cyclically hyperbolic group cannot contain (virtually) simple groups. 

\begin{cor}
Let $G$ be a group acting geometrically on a CAT(0) cube complex. If $G$ is cyclically hyperbolic, then every infinite subgroup in $G$ contains infinitely many normal subgroups.
\end{cor}

\begin{proof}
 Let $H \leq G$ be an infinite subgroup. According to Theorem~\ref{thm:TitsAlt}, either $H$ is virtually free abelian or it contains a finite-index $\dot{H}$ whose commutator subgroup $[\dot{H},\dot{H}]$ is acylindrically hyperbolic group. In the former case, $H$ contains infinitely many finite quotients, so the desired conclusion holds. Similarly, if $[\dot{H},\dot{H}]$ has infinite index in $\dot{H}$, then $\dot{H}$ surjects onto a infinite abelian group, so $\dot{H}$, and a fortiori $H$, contains infinitely many normal finite-index subgroups. It remains to consider the case where $[\dot{H},\dot{H}]$ is acylindrically hyperbolic and has finite-index in $\dot{H}$. But then $[\dot{H},\dot{H}]$ is SQ-universal, which implies that $H$ must be SQ-universal as well, from which we conclude that $H$ contains infinitely many normal subgroups. 
%Let $H \leq G$ be an infinite subgroup. If $H$ is virtually free abelian, then the conclusion is clear since $H$ contains infinitely many (finite) quotients. From now on, assume that $H$ is not virtually free abelian. Let $H=H_0 \geq H_1 \geq \cdots \geq H_k$ be the subgroups given by Theorem~\ref{thm:TitsAlt}. If $H_i/H_{i+1}$ is finite for every $0 \leq i \leq k-1$, then $H_k$ has finite index in $H$, which implies that $H$ is SQ-universal, providing the desired conclusion. Otherwise, let $0 \leq i \leq k-1$ denote the first index such that $H_i/H_{i+1}$ is infinite. Then $H$ contains a finite-index subgroup, namely $H_i$, that surjects onto $\mathbb{Z}$. This implies that $H$ has infinitely many (finite) quotients. 
\end{proof}

\noindent
We can also deduce from Theorem~\ref{thm:TitsAlt} that one of the two implications given by Theorem~\ref{thm:Special} is true for all cubulable groups.

\begin{cor}\label{cor:F2xF2}
A cyclically hyperbolic group that acts geometrically on a CAT(0) cube complex does not contain a subgroup isomorphic to $\mathbb{F}_2 \times \mathbb{F}_2$. 
\end{cor}

\begin{proof}
 Assume for contradiction that $\mathbb{F}_2 \times \mathbb{F}_2$ is a subgroup of a cubulable cyclically hyperbolic group. According to Theorem~\ref{thm:TitsAlt}, there exists a finite-index subgroup $H \leq \mathbb{F}_2 \times \mathbb{F}_2$ whose commutator subgroup $[H,H]$ is acylindrically hyperbolic. Setting $H_1:= H \cap (\mathbb{F}_2 \times \{1\})$ and $H_2:= H \cap (\{1\} \times \mathbb{F}_2)$, we obtain a finite-index subgroup $H_1 \times H_2$ in $H$. Notice that $H_1$ and $H_2$ are two non-abelian free subgroups that are normal in $H$. Since $[H_1,H_1] \times [H_2,H_2]$ is an infinite normal subgroup in $[H,H]$, it has to be acylindrically hyperbolic (see \cite[Lemma~7.2]{OsinAcyl}), which is impossible since an acylindrically hyperbolic group does not decompose as a product of two infinite groups (see \cite[Corollary~7.3]{OsinAcyl}). 
%Assume for contradiction that $\mathbb{F}_2 \times \mathbb{F}_2$ is a subgroup of a cubulable cyclically hyperbolic group. Let $H_0 \geq H_1 \geq \cdots \geq H_k$ be the subgroups of $\mathbb{F}_2 \times \mathbb{F}_2$ as given by Theorem~\ref{thm:TitsAlt}. We claim that, for every $0 \leq i \leq k$, $H_i$ contains a subgroup $K_i$ that decomposes as a product $A_i \times B_i$ of two free groups of infinite rank that are normal in $H_i$. For $i=0$, we define $K_i$ as the derived subgroup of $\mathbb{F}_2 \times \mathbb{F}_2$. Now, assume that our claim holds for some $0 \leq i \leq k-1$. Two cases may happen.
%\begin{itemize}
%	\item If $H_i/H_{i+1}$ is finite, then we define $A_{i+1}:=(A_i \cap H_{i+1})$ and $B_{i+1}:= (B_i \cap H_{i+1})$. 
%	\item If $H_i/H_{i+1}$ is free abelian, then $H_{i+1}$ contains the desired subgroup $D(K_i)=D(A_i) \times D(B_i)$. Because $D(A_i)$ is characteristic in $A_i$ and that $A_i$ is normal in $H_i$, we know that $D(A_i)$ is normal in $H_{i+1}$. Similarly, $D(B_i)$ is normal in $H_{i+1}$. Therefore, we set $A_{i+1}:=A_i$ and $B_{i+1}:=B_i$. 
%\end{itemize}
%We conclude that $H_k$ contains a normal subgroup isomorphic to a product of two infinite free groups. This is impossible according to \cite[Lemma~7.2 and Corollary~7.3]{OsinAcyl}.
\end{proof}

\noindent
However, it is worth noticing that the converse does not hold: there exist BMW-groups that do not contain products of non-abelian free groups, even though they cannot be cyclically hyperbolic. See Section~\ref{section:Examples} for more details. 

\medskip \noindent
 We emphasize that we do not know whether it can be deduced from Theorem~\ref{thm:TitsAlt} that every subgroup of a cubulable cyclically hyperbolic group is either virtually free abelian or SQ-universal, i.e.\ that Corollary~\ref{cor:SQuniversal} can be generalised to arbitrary subgroups. Nevertheless, we expect this statement to be true: 
%We emphasize that, in our opinion, Theorem~\ref{thm:TitsAlt} is not optimal. Indeed, we expect that Corollary~\ref{cor:SQuniversal} can be generalised to subgroups, leading to:

\begin{conj}
Let $G$ be a group acting geometrically on a CAT(0) cube complex. If $G$ is cyclically hyperbolic, then every subgroup in $G$ is either virtually abelian or SQ-universal.
\end{conj}

\noindent
As a consequence of Corollary~\ref{cor:F2xF2}, in order to show that a given cubulable group does not contain $\mathbb{F}_2 \times \mathbb{F}_2$ as a subgroup, it suffices to show that it is cyclically hyperbolic. The latter property may be easier to verify because there exist efficient criteria allowing us to recognize hyperbolic graphs. In the final section of the article, we illustrate this idea by considering a few explicit families of groups, including right-angled Coxeter groups and graph braid groups. For instance, we prove that:

\begin{cor}\label{cor:ProductInRACG}
Let $\Gamma$ be a simplicial graph. The right-angled Coxeter group $C(\Gamma)$ contains a subgroup isomorphic to $\mathbb{F}_2 \times \mathbb{F}_2$ if and only if $\Gamma$ contains an induced subgraph isomorphic to $K_{3,3}$, $K_{3,3}^+$, or $K_{3,3}^{++}$. 
\end{cor}

\noindent
The graphs $K_{3,3}$, $K_{3,3}^+$, and $K_{3,3}^{++}$ are illustrated by Figure~\ref{Graphs} in Section~\ref{section:Examples}. Notice that, by using the embedding constructed in \cite{MR1783167} of a right-angled Artin group as a finite-index subgroup in a right-angled Coxeter group, we recover the fact proved in \cite{MR2475886} according to which a right-angled Artin group contains $\mathbb{F}_2 \times \mathbb{F}_2$ as a subgroup if and only if its defining graph contains an induced $4$-cycle.

\paragraph{A few words about the proofs.} In Section~\ref{section:Obstruction}, we prove the main criterion that allows us to deduce that some cubulable groups are not cyclically hyperbolic. More precisely, Proposition~\ref{prop:NotHyp} states that, if the CAT(0) cube complex on which our group $G$ acts geometrically contains two transverse \emph{trees of hyperplanes} satisfying some conditions (i.e.\ a configuration of hyperplanes that mimics the combinatorics of hyperplanes in a product of two bushy trees), then $G$ cannot be cyclically hyperbolic. From this criterion, combined with the rank one conjecture proved in \cite{MR2827012} and the connection between rank one isometries and acylindrically hyperbolicity \cite{MR3849623, BBF, MR4057355}, Theorem~\ref{thm:Structure} follows easily. 

\medskip \noindent
Theorem~\ref{thm:TitsAlt} follows the same strategy, but one additional ingredient, of independent interest, is needed. In Section~\ref{section:Horo}, we simplify and generalise the morphisms constructed in the appendix of \cite{MR3509968}. More precisely, given a finite-dimensional CAT(0) cube complex $X$, we associate to each component $Y \subset \mathfrak{R}X$ in its Roller boundary a \emph{(full) horomorphism} $\mathfrak{H}_Y$, defined on a finite-index subgroup of $\mathrm{stab}(Y) \leq \mathrm{Isom}(X)$ and taking values in a free abelian group of finite rank. The key property is that an element in the kernel of $\mathfrak{H}_Y$ is elliptic in $Y$ if and only if it is elliptic in $X$. See Theorem~\ref{thm:KernelFullHoro}. We deduce from these morphisms the following statement (see Theorem~\ref{thm:Devissage}):

\begin{thm}\label{thm:IntroDevissage}
 Let $G$ be a group acting on a finite-dimensional CAT(0) cube complex $X$. There exist a finite-index subgroup $H \leq G$ and a convex $H$-invariant subcomplex $Z \subset \overline{X}$ such that the following assertions hold:
\begin{itemize}
	\item if $[H,H]$ is infinite, then it acts essentially on $Z$ without virtually stabilising a component of $\mathfrak{R}Z$;
	\item every $Z$-elliptic element in $[H,H]$ is $X$-elliptic.
\end{itemize} 
\end{thm}

\noindent
The central point is that the action $[H,H] \curvearrowright Z$ is as good as $G \curvearrowright X$ (for instance, if $G$ acts properly on $X$, then so does $[H,H]$ on $Z$) and that the machinery developed in \cite{MR2827012} applies to $[H,H] \curvearrowright Z$. In particular, if $G$ is a subgroup of a cyclically hyperbolic group acting geometrically on $X$, then $Z$ has to be irreducible and non-Euclidean (if unbounded) and $[H,H]$ acts properly on it, which implies that $[H,H]$ is acylindrically hyperbolic. This is how Theorem~\ref{thm:TitsAlt} is proved.

\medskip \noindent
Because Corollary~\ref{cor:F2xF2} follows from Theorem~\ref{thm:TitsAlt}, one direction of Theorem~\ref{thm:Special} is proved. In order to prove the converse, namely in order to prove that a cocompact special group with no $\mathbb{F}_2 \times \mathbb{F}_2$ is cyclically hyperbolic, we begin by proving a general criterion that allows us to show that some cone-offs of CAT(0) cube complexes are hyperbolic. We already proved in \cite{coningoff} that a cone-off over convex subcomplexes in which large flat rectangles become bounded is automatically hyperbolic. However, this criterion is not sufficient for our purpose because flats might remain unbounded in the cone-off. (For instance, in order to show that $\mathbb{F}_2 \times \mathbb{Z}$ is cyclically hyperbolic, we need to glue cones over the cosets of $\mathbb{Z}$. In the cone-off, which is a quasi-tree, flats collapse to bi-infinite lines and remain unbounded.) We improve \cite[Theorem~4.1]{coningoff} in Section~\ref{section:ConeOff} by proving that a cone-off over isometrically embedded subcomplexes in which large flat rectangles collapse along one of their coordinates is automatically hyperbolic; see Theorem~\ref{thm:WhenHyp}. Our proof relies on Lemma~\ref{lem:Interval}, dedicated to the structure of intervals in CAT(0) cube complexes and which is of independent interest. Finally, the combination of this hyperbolicity criterion with the formalism we introduced in \cite{SpecialRH} leads to the proof of Theorem~\ref{thm:Special}.

\paragraph{Acknowledgements.} I am grateful to Pierre-Emmanuel Caprace for bringing \cite{MR2174099} to my attention, and to the anonymous referee for their indication to an improvement of Theorem~\ref{thm:Devissage}, which leads to an improvement of Theorem~\ref{thm:TitsAlt}.

\section{Constructing hyperbolic cone-offs}\label{section:ConeOff}

\noindent
 Recall that, given a graph $X$ and a collection of subgraphs $\mathcal{P}$, the \emph{cone-off of $X$ over $\mathcal{P}$} is the graph obtained from $X$ by adding a new vertex $v_P$ for every $P \in \mathcal{P}$ and by connecting each $v_P$ to all the vertices of $P$ with an edge. 

\medskip \noindent
In order to show that a given finitely generated group $G$ is cyclically hyperbolic, we need to find finitely many elements $g_1,\ldots, g_n \in G$ such that coning-off a Cayley graph of $G$ over the cosets of $\langle g_1 \rangle, \ldots, \langle g_n \rangle$ produces a hyperbolic space. In fact, as a consequence of the following well-known statement, the Cayley graph of $G$ can be replaced with any geodesic metric space on which $G$ acts geometrically and the cosets of the $\langle g_i \rangle$ can be replaced with subspaces on which they act cocompactly.

\begin{lemma}\label{lem:MilnorSvarc}
Let $G$ be a group acting properly and cocompactly on a geodesic metric space $X$. Let $\mathcal{H}$ be a finite collection of subgroups, and, for every $H \in \mathcal{H}$, fix a subspace $X_H \subset X$ on which $H$ acts cocompactly. Then $G$ is finitely generated, and, given a basepoint $x_0 \in X$ and a finite generating set $S \subset G$, the orbit map $G \to X$ given by $g \mapsto g \cdot x_0$ induces a quasi-isometry from the cone-off of $\mathrm{Cayl}(G,S)$ over the cosets of subgroups in $\mathcal{H}$ to the cone-off of $X$ over the subspaces in $\{gX_H \mid H \in \mathcal{H},g \in G\}$. 
\end{lemma}

\noindent
(For more details, we refer, for instance, to the classical Milnor-\v Svarc lemma  (e.g.\ \cite[Proposition~5.3.6]{MR3729310})  combined with \cite[Lemmas~5.3 and~5.5]{Excentric}.)

\medskip \noindent
Since we focus on groups acting geometrically on CAT(0) cube complexes, it is natural to ask when the cone-off of a CAT(0) cube complex is hyperbolic. We already proved a sufficient criterion in \cite[Theorem~4.1]{coningoff}, but it is not sufficient for our purpose here, because the subcomplexes on which we want to cone-off (typically, axes of isometries) may not be convex and because flat rectangles (i.e.\ isometric embeddings from $[0,a] \times [0,b]$ for some $a,b \geq 0$) may remain unbounded after coning-off. In this section, our goal is to improve this criterion.

\medskip \noindent
In the sequel, we assume that the reader is already familiar with CAT(0) cube complexes. We emphasize that distances and geodesics between vertices refer to the graph metric in the one-skeleton.

\subsection{Disc diagrams}

\noindent
This subsection is dedicated to preliminary definitions and results about \emph{disc diagrams}.

\begin{definition}
Let $X$ be a nonpositively curved cube complex. A \emph{disc diagram} is a continuous combinatorial map $D \to X$, where $D$ is a finite contractible square complex with a fixed topological embedding into $\mathbb{S}^2$; notice that $D$ may be \emph{degenerate}, i.e.\ not homeomorphic to a disc. The complement of $D$ in $\mathbb{S}^2$ is a $2$-cell, whose attaching map will be referred to as the \emph{boundary path} $\partial D \to X$ of $D \to X$; it is a combinatorial path. The \emph{area} of $D \to X$, denoted by $\mathrm{Area}(D)$, corresponds to the number of squares of $D$.
\end{definition}

\noindent
Given a combinatorial closed path $P \to X$, a disc diagram $D \to X$ is \emph{bounded} by $P \to X$ if there exists an isomorphism $P \to \partial D$ such that the following diagram is commutative:
\begin{displaymath}
\xymatrix{ \partial D \ar[rr] & & X \\ P \ar[u] \ar[urr] & & }
\end{displaymath}
As a consequence of a classical argument due to van Kampen \cite{vanKampen} (see also \cite[Lemma~2.17]{McCammondWise}), there exists a disc diagram bounded by a given combinatorial closed path if and only if this path is null-homotopic. Thus, if $X$ is a CAT(0) cube complex, then any combinatorial closed path bounds a disc diagram. As a square complex, a disc diagram contains hyperplanes: they are called \emph{dual curves}. Equivalently, they correspond to the connected components of the reciprocal images of the hyperplanes of $X$. 

\medskip \noindent
Given a CAT(0) cube complex $X$, a \emph{cycle of subcomplexes} $\mathcal{C}$ is a sequence of subcomplexes $C_1, \ldots, C_r$ such that $C_1 \cap C_r \neq \emptyset$ and $C_i \cap C_{i+1} \neq \emptyset$ for every $1 \leq i \leq r-1$. A disc diagram $D \to X$ is \emph{bounded} by $\mathcal{C}$ if $\partial D \to X$ can be written as the concatenation of $r$ combinatorial geodesics $P_1, \ldots, P_r \to X$ such that $P_i \subset C_i$ for every $1 \leq i \leq r$. The \emph{complexity} of such a disc diagram is defined by the couple $(\mathrm{Area}(D), \mathrm{length}(\partial D))$, and a disc diagram bounded by $\mathcal{C}$ will be of \emph{minimal complexity} if its complexity is minimal with respect to the lexicographic order among all the possible disc diagrams bounded by $\mathcal{C}$ (allowing modifications of the paths $P_i$). It is worth noticing that such a disc diagram does not exist if our subcomplexes contain no combinatorial geodesics. On the other hand, if our subcomplexes are combinatorially geodesic, then a disc diagram always exists.

\begin{prop}\label{prop:DiscDiag}\emph{(\cite[Theorem 2.13]{coningoff})}
Let $X$ be a CAT(0) cube complex, $\mathcal{C}=(C_1, \ldots, C_r)$ a cycle of subcomplexes, and $D \to X$ a disc diagram bounded by $\mathcal{C}$. For convenience, write $\partial D$ as the concatenation of $r$ combinatorial geodesics $P_1, \ldots, P_r \to X$ with $P_i \subset C_i$ for every $1 \leq i\leq r$. If the complexity of $D \to X$ is minimal, then:
\begin{itemize}
	\item[(i)] if $C_i$ is combinatorially convex, two dual curves intersecting $P_i$ are disjoint;
	\item[(ii)] if $C_i$ and $C_{i+1}$ are combinatorially convex, no dual curve intersects both $P_i$ and $P_{i+1}$.
\end{itemize}
\end{prop}

\noindent
In general, a disc diagram $D \to X$ is not an embedding. Nevertheless:

\begin{prop}\label{prop:DiscEmbedding}\emph{(\cite[Proposition 2.15]{coningoff})}
Let $X$ be a CAT(0) cube complex and $D \to X$ a disc diagram which does not contain any bigon. With respect to the combinatorial metrics, $\varphi : D \to X$ is an isometric embedding if and only if every hyperplane of $X$ induces at most one dual curve of $D$. 
\end{prop}
\begin{figure}
\begin{center}
\includegraphics[width=0.6\linewidth]{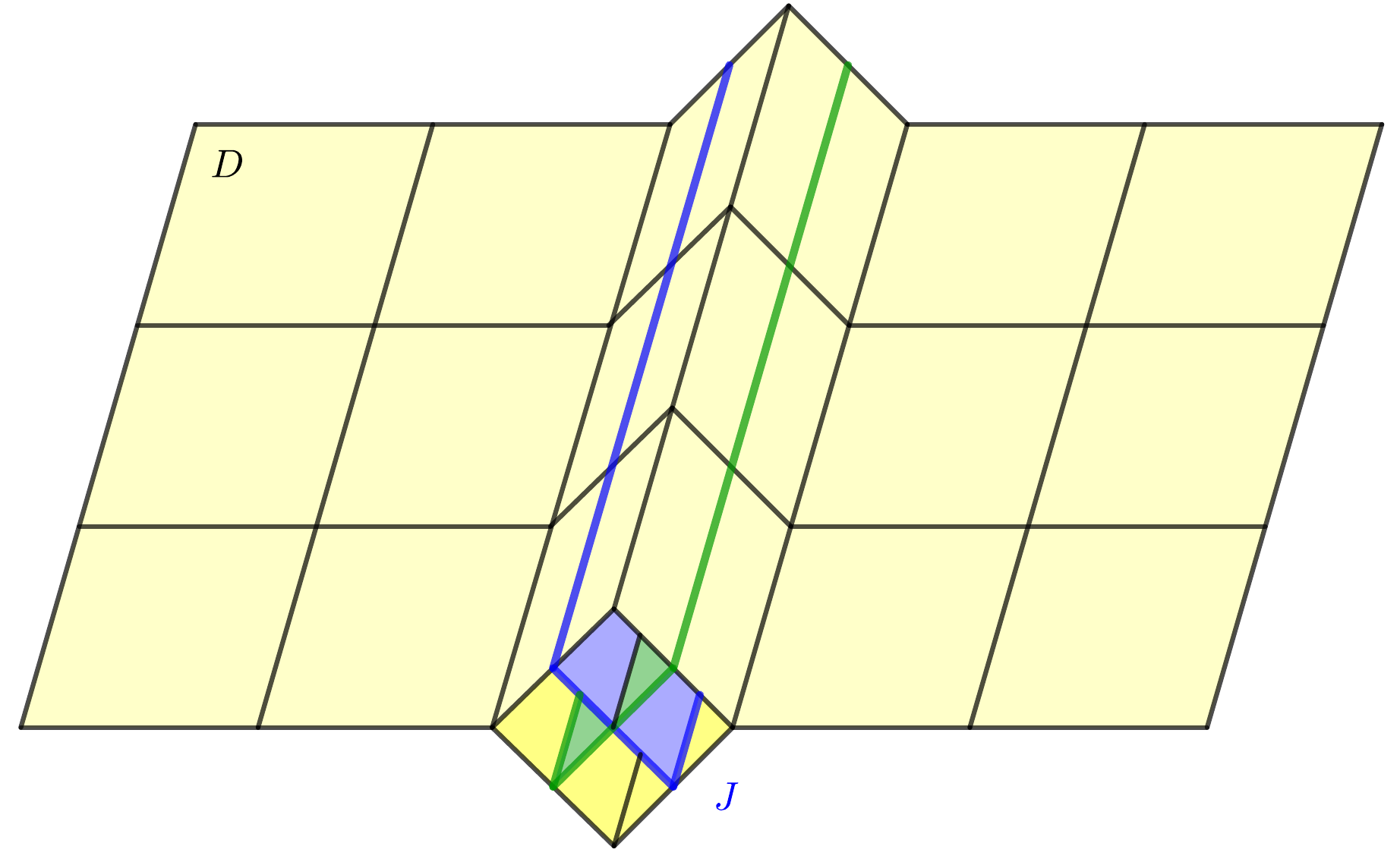}
\caption{Shifting the hyperplane $J$ by modifying the disc diagram.}
\label{DiagramFlip}
\end{center}
\end{figure}

\noindent
We conclude this section with an observation that we will use in Section~\ref{section:Obstruction}. Recall that, in a CAT(0) cube complex $X$, a \emph{flat rectangle} is an isometric embedding $[0,a] \times [0,b]$ for some $a,b \geq 0$. 

\begin{prop}\label{prop:FlatRectangles}
Let $X$ be a CAT(0) cube complex, $\mathcal{C}=(C_1,C_2,C_3,C_4)$ be a cycle of four subcomplexes, and $D \to X$ a disc diagram of minimal complexity bounded by $\mathcal{C}$. Then $D \to X$ is a flat rectangle. Moreover, every hyperplane crossing the part of $\partial D \to X$ lying in $C_1$ separates $C_2$ and $C_4$. 
\end{prop}

\begin{proof}
The fact that $D \to X$ is a flat rectangle is an easy consequence of Propositions~\ref{prop:DiscDiag} and~\ref{prop:DiscEmbedding}. See \cite[Corollary~2.17]{coningoff} for more details. For simplicity, we identify $D$ with $[0,a] \times [0,b]$ for some $a,b \geq 0$ where $[0,a] \times \{0\} \subset C_1$. Also, we identify $D$ with its image under $D \to X$. We need to show that every hyperplane $J$ crossing $[0,a] \times\{0\}$ separates $C_2$ and $C_4$. Assume for contradiction that it is not the case. So $J$ crosses $C_2$ or $C_4$, say $C_2$. Without loss of generality, we assume that $J$ is the leftmost such hyperplane along $[0,a] \times \{0\}$. As a consequence, $J$ is transverse to all the hyperplanes crossing $[0,a] \times \{0\}$ on its left. As shown by Figure~\ref{DiagramFlip}, we can modify the disc diagram without modifying its area nor its perimeter but shifting $J$ to the left. Therefore, we can assume that $J$ crosses $[0,1] \times \{0\}$. But the fact that $J$ crosses $C_2$, and because $C_2$ is convex, implies that $\{1\} \times [0,b]$ is contained $C_2$. Indeed, if $z$ denotes an arbitrary vertex of $C_2$ separated from $\{0\} \times [0,b]$ by $J$, then, for every $0 \leq p \leq b$, there exists a geodesic from $(0,p)$ to $z$ passing through $(1,p)$, hence $(1,p) \in C_2$ by convexity. Thus, because $[0,1] \times [0,b]$ is contained in $C_2$, we can replace our disc diagram $[0,a] \times [0,b]$ with $[1,a] \times [0,b]$, contradicting the minimality of its complexity. This completes the proof of our proposition.
\end{proof}

\subsection{A lemma on intervals}

\noindent
The first step in the proof of \cite[Theorem~4.1]{coningoff} was to show that, given two geodesics in a CAT(0) cube complex with the same endpoints, every vertex of one geodesic is the corner of a flat square that has its opposite vertex on the other geodesic (see Figure~\ref{staircase}). We improve this observation by constructing a \emph{staircase} instead of a square.
\begin{figure}
\begin{center}
\includegraphics[scale=0.4]{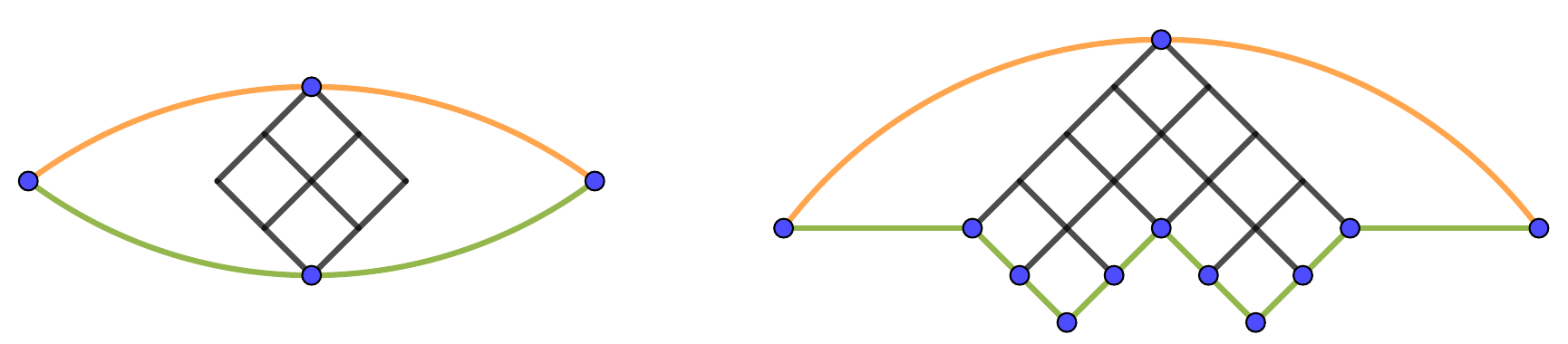}
\caption{A comparison between \cite{coningoff} and Lemma~\ref{lem:Interval}.}
\label{staircase}
\end{center}
\end{figure}

\begin{definition}
A \emph{staircase} is a square complex isomorphic to the subcomplex in $\mathbb{R}^2$ delimited by $[0,a] \times \{0\}$, $\{0\} \times [0,b]$ and some ($\ell^1$-)geodesic $\alpha$ between $(a,0)$ and $(0,b)$ where $a,b \geq 0$. Its \emph{corner} refers to the vertex corresponding to $(0,0)$, and its \emph{broken path} as the geodesic corresponding to $\alpha$.
\end{definition}

\noindent
The next observation will be fundamental in the proof of Theorem~\ref{thm:WhenHyp} below, and is of independent interest.

\begin{lemma}\label{lem:Interval}
Let $X$ be a CAT(0) cube complex, $[x,y]$ a geodesic between two vertices $x,y \in X$, and $z \in I(x,y)$ a third vertex. Then there exists an isometrically embedded staircase $E \subset X$ such that $z$ is its corner and such that its broken path is contained in~$[x,y]$.
\end{lemma}

\noindent
Here, given two vertices $a,b \in X$, $I(a,b)$ refers to the \emph{interval} between $a$ and $b$, i.e.\ the union of all the geodesics between $a$ and $b$.

\begin{proof}[Proof of Lemma~\ref{lem:Interval}.]
Let $D \to X$ be a disc diagram of minimal area bounded by the cycle of subcomplexes $I(x,z) \cup I(z,y) \cup \gamma$. Because a hyperplane in $X$ cannot cross twice a geodesic, necessarily no two dual curves in $D$ are sent to the same hyperplane in $X$. It follows from Proposition~\ref{prop:DiscEmbedding} that $D \to X$ is an isometric embedding. From now on, we identify $D$ with its image in $X$. The boundary of $D$ is the concatenation of a subsegment $[x',y'] \subset \gamma$, a geodesic $[y',z]$, and a geodesic $[z,x']$. According to Proposition~\ref{prop:DiscDiag}, no two dual curves in $D$ crossing $[z,x']$ (resp. $[z,y']$) are transverse and no dual curve crosses both $[z,x']$ and $[z,y']$. As a consequence, $D$ is a CAT(0) cube complex on its own right (otherwise, it would contain three squares intersecting cyclically, and so three pairwise intersecting dual curves, two of them having to intersect $[z,x']$ or $[z,y']$) and the geodesics $[z,x']$, $[z,y']$ are convex in $D$ (otherwise, $D$ would contain a square having two consecutive sides on $[z,x']$ or on $[z,y']$, producing two intersecting dual curves both crossing $[z,x']$ or both crossing $[z,y']$). Notice that a vertex in $D$ is uniquely determined by its projections onto $[z,x']$ and $[z,y']$. Indeed, two distinct vertices $p,q \in D$ must be separated by a hyperplane, and such a hyperplane must cross either $[z,x']$ or $[z,y']$ and so it has to separate the projections of $p,q$ onto $[z,x']$ or onto $[z,y']$. Therefore, we get an isometric embedding 
$$\left\{ \begin{array}{ccc} D & \hookrightarrow & [0,d(z,x')] \times [0,d(z,y')] \\ p & \to & \left( d(z, \mathrm{proj}_{[z,x']}(p)), d(z, \mathrm{proj}_{[z,y']}(p)) \right) \end{array} \right..$$
The image of $D$ is delimited by $[0,d(z,x')] \times \{0\}$, $\{0\} \times [0,d(z,y')]$, and the image of $[x',y']$, which connects $(d(z,x'),0)$ and $(0,d(z,y'))$. Consequently, $D$ is the staircase we are looking for.
\end{proof}

\subsection{A hyperbolicity criterion}

\noindent
We are now ready to state and prove our improvement of \cite[Theorem~4.1]{coningoff}, which gives a sufficient condition for the coning-off of a CAT(0) cube complex to be hyperbolic:

\begin{thm}\label{thm:WhenHyp}
Let $X$ be a CAT(0) cube complex. Fix a collection $\mathcal{P}$ of isometrically embedded subcomplexes and let $Y$ denote the corresponding cone-off of $X$. Assume that there exists a constant $K \geq 0$ such that, for every flat rectangle $[0,a] \times [0,b] \hookrightarrow X$, either the Hausdorff distance in $Y$ between $[0,a] \times \{i\}$ and $[0,a] \times \{j\}$ is $\leq K$ for all $0 \leq i,j \leq b$, or the Hausdorff distance in $Y$ between $\{i\} \times [0,b]$ and $\{j\} \times [0,b]$ is $\leq K$ for all $0 \leq i,j \leq a$. Then $Y$ is hyperbolic.
\end{thm}

\noindent
Our proof follows the lines of \cite[Theorem~4.1]{coningoff} and uses the following criterion:

\begin{prop}\label{prop:WhenHyp}\emph{(\cite[Proposition~3.1]{Bowditchcriterion})}
For every $D \geq 0$, there exists some $\delta \geq 0$ such that the following holds. Let $T$ be a graph. Assume that a connected subgraph $\eta(x,y)$, containing $x$ and $y$, is associated to each pair of vertices $(x,y) \in T^2$ such that:
\begin{itemize}
	\item for all vertices $x,y \in T$, $d(x,y) \leq 1$ implies $\mathrm{diam}(\eta(x,y)) \leq D$;
	\item for all vertices $x,y,z \in T$, the inclusion $\eta(x,y) \subset \left( \eta(x,z) \cup \eta(z,y) \right)^{+D}$ holds.
\end{itemize}
Then $T$ is $\delta$-hyperbolic.
\end{prop}

\noindent
 Here, given a metric space $X$, a subset $Y \subset X$, and a constant $D \geq 0$, we denote by $Y^{+D}$ the $D$-neighbourhood of $Y$, i.e.\ $Y^{+D}:= \{ x \in X \mid d(x,Y) \leq D \}$.  

\begin{proof}[Proof of Theorem~\ref{thm:WhenHyp}.]
The strategy is to define $\eta(x,y)$ as the subgraph in $Y$ generated by the interval $I(x,y)$ in $X$ for all $x,y \in X$ and then to apply Proposition~\ref{prop:WhenHyp}. In order to verify the assumptions, the following observation will be needed:

\begin{claim}\label{claim:ForHyp}
Let $x,y \in X$ be two vertices, $[x,y]$ a geodesic in $X$, and $z \in I(x,y)$ a third vertex. Then $z$ lies in the $2K$-neighbourhood of $[x,y]$ in $Y$.
\end{claim}

\noindent
According to Lemma \ref{lem:Interval}, there exists an isometrically embedded staircase $E \subset X$ such that $z$ is its corner and such that its broken path $[x',y']$ lies in $[x,y]$. Write $E$ as a union of flat rectangles $C_1, \ldots, C_r$ such that, for every $1 \leq i \leq r$, $z$ is a corner of $C_i$ and its opposite vertex lies in $[x,y]$. We index our rectangles so that the length of $C_i \cap [z,x']$ decreases. See Figure \ref{ColorStaircase}. We refer to a geodesic in $E$ parallel to $[z,x']$ (resp. to $[z,y']$) as \emph{vertical} (resp. \emph{horizontal}). By extension, we say that a flat rectangle in $E$ is \emph{vertical} (resp. \emph{horizontal}) if the Hausdorff distance in $Y$ between any two of its vertical (resp. horizontal) geodesics is $\leq K$. By assumption, each $C_i$ is vertical or horizontal. Four cases may happen:
\begin{itemize}
	\item If there exists some $1 \leq i \leq r$ such that $C_i$ is both vertical and horizontal, then the distance in $Y$ between $z$ and the vertex of $C_i$ opposite to it (which belongs to $[x,y]$) must be $\leq 2K$. 
	\item If $C_1, \ldots, C_r$ are all vertical, then in particular $C_r$ is vertical, which implies that $d_Y(z,y') \leq K$.
	\item If $C_1, \ldots, C_r$ are all horizontal, then in particular, $C_1$ is horizontal, which implies that $d_Y(z,x') \leq K$.
	\item If there exists some $1 \leq i \leq r-1$ such that $C_i$ is horizontal and $C_{i+1}$ vertical, or vice-versa, then the flat rectangle $C_i \cap C_{i+1}$ is both vertical and horizontal, so the distance in $Y$ between $z$ and the vertex of $C_i \cap C_{i+1}$ opposite to it (which belongs to $[x,y]$) must be $\leq 2K$.
\end{itemize}
This concludes the proof of our claim.
\begin{figure}
\begin{center}
\includegraphics[scale=0.4]{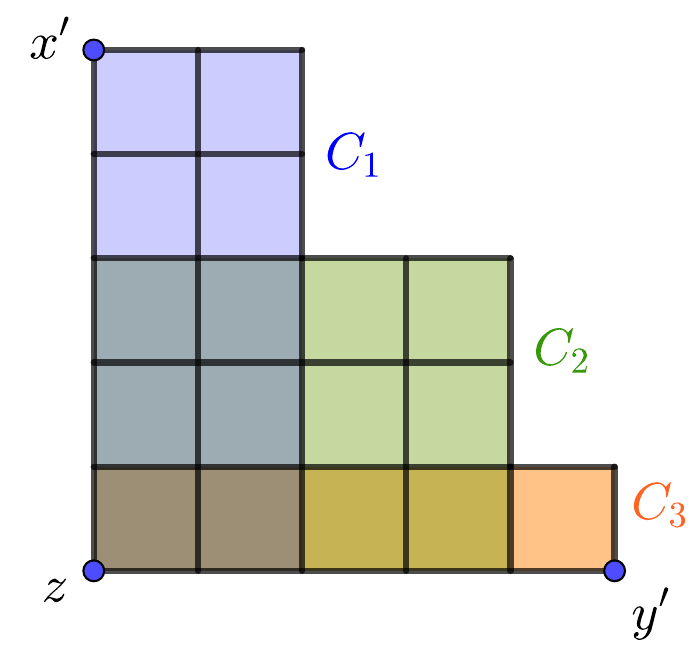}
\caption{Decomposition of the staircase $E$ as a union of flat rectangles.}
\label{ColorStaircase}
\end{center}
\end{figure}

\medskip \noindent
We are now ready to verify that Proposition~\ref{prop:WhenHyp} applies. First, let $x,y \in Y$ be two vertices satisfying $d_Y(x,y) \leq 1$. If $x=y$, then $\eta(x,y)$ is a single vertex; and if $x$ and $y$ are adjacent in $X$, then $\eta(x,y)$ is a single edge. Otherwise, if $d_X(x,y) \geq 2$, then there must exist some $P \in \mathcal{P}$ such that $x,y \in P$. Because $P$ is isometrically embedded, there exists a geodesic $\gamma \subset P$ between $x$ and $y$, and, as a consequence of Claim~\ref{claim:ForHyp}, $d_Y(z, \gamma) \leq 2K$. Therefore, 
$$\mathrm{diam}_Y(\eta(x,y)) \leq 4K+ \mathrm{diam}_Y(\gamma) =4K+1.$$
Thus, we have proved that $\mathrm{diam}_Y(\eta(x,y)) \leq 4K+1$ for all vertices $x,y \in Y$ satisfying the inequality $d_Y(x,y) \leq 1$. 

\medskip \noindent
Next, fix three vertices $x,y,z \in Y$ and a vertex $w \in \eta(x,y)$. Let $m$ denote the median point of $\{x,y,z\}$ and let $[m,x]$, $[m,y]$ be two geodesics in $X$. Because $[x,m] \cup [m,y]$ defines a geodesic in $X$, it follows from Claim~\ref{claim:ForHyp} that there exists a vertex $w' \in [x,m] \cup [m,y]$ such that $d_Y(w,w') \leq 2K$. Because $[x,m] \subset \eta(x,z)$ and $[m,y] \subset \eta(y,z)$, it follows that $z$ lies in the $2K$-neighbourhood of $\eta(x,z) \cup \eta(y,z)$ in $Y$, as desired. 

\medskip \noindent
Thus, Proposition~\ref{prop:WhenHyp} applies and we conclude that $Y$ is hyperbolic.
\end{proof}

\section{An obstruction to cyclic hyperbolicity}\label{section:Obstruction}

\noindent
In the opposite direction of Section~\ref{section:ConeOff}, we need to be able to show that some groups acting geometrically on CAT(0) cube complexes are not cyclically hyperbolic. The typical example to keep in mind is $\mathbb{F}_2 \times \mathbb{F}_2$ that acts on a product of two regular trees. As a warm up, let us sketch a proof that $\mathbb{F}_2 \times \mathbb{F}_2$ is not cyclically hyperbolic.

\medskip \noindent
Write $\mathbb{F}_2 \times \mathbb{F}_2 = \langle a,b \mid \ \rangle \times \langle c,d \mid \ \rangle$ and let $g_1, \ldots, g_n \in \mathbb{F}_2 \times \mathbb{F}_2$ be finitely many cyclically reduced elements.  Write each $g_i$ as $(\ell_i,r_i)$.  Fix an infinite geodesic ray $r_1$ in $\mathbb{F}_2 \times \{1\}$ such that the projection on $r_1$ of the axis of any conjugate of a $g_i$ has diameter at most $K$, for some uniform $K \geq 0$. This amounts to constructing an infinite reduced word written over $\{a,b\}^{\pm 1}$ sharing no large subwords with  the infinite word $\ell_i^\infty:= \ell_i \ell_i \cdots$ . Similarly, fix an infinite geodesic ray $r_2$ in $\{1\} \times \mathbb{F}_2$ such that the projection on $r_2$ of the axis of any conjugate of a $g_i$ has diameter at most $K$ (up to replacing $K$ with a larger constant). Then $r_1 \times r_2$ is a convex subcomplex isometric to $[0,+ \infty) \times [0,+ \infty)$ with a the property that the projection on $r_1 \times r_2$ of the axis of any conjugate of a $g_i$ has diameter $\leq 2K$. This implies that the flat quasi-isometrically embeds into the cone-off of $\mathbb{F}_2 \times \mathbb{F}_2$ over the axes of the conjugates of the $g_i$, proving that $\mathbb{F}_2 \times \mathbb{F}_2$ is not weakly hyperbolic relative to $\{g_1, \ldots, g_n\}$. 

\medskip \noindent
In this section, our goal is to generalise this observation. Roughly speaking, we want to show that, given a group acting geometrically on a CAT(0) cube complex, as soon as there is a subcomplex that ``looks like'' a product of two regular trees, then the group is not cyclically hyperbolic.

\subsection{Transverse trees of hyperplanes}

\noindent
In order to state our criterion, we need to introduce some terminology:

\begin{definition}
Let $X$ be a CAT(0) cube complex. A \emph{tree of hyperplanes} is the data of a rooted tree $T$ and a collection of pairwise disjoint hyperplanes $\mathcal{H}= \{ H_v \mid \in V(T) \}$ indexed by the vertices of $T$ such that, for all vertices $u,v,w \in T$, $H_v$ separates $H_u$ and $H_w$ if and only if $u,v,w$ lie in that order on an infinite ray starting from the root of $T$. A tree of hyperplanes 
\begin{itemize}
	\item is \emph{branching} if every vertex in the underlying rooted tree has at least two children; 
	\item is \emph{uniformly distributed} if there exists some constant $K\geq 0$ such that, for all vertices $u,v \in T$, the inequality $d(H_u,H_v) \leq K \cdot d_T(u,v)$ holds;
	\item has no \emph{transeparation} if every hyperplane separating at least two hyperplanes in $\mathcal{H}$ cannot be transverse to two hyperplanes in $\mathcal{H}$.
\end{itemize}
Two trees of hyperplanes $\mathcal{H}_1,\mathcal{H}_2$ are \emph{transverse} if every hyperplane in $\mathcal{H}_1$ is transverse to every hyperplane in $\mathcal{H}_2$.
\end{definition}

\noindent
For instance, a regular tree contains many trees of hyperplanes; and a product of regular trees contains many transverse trees of hyperplanes. Our main result is the following criterion:

\begin{prop}\label{prop:NotHyp}
Let $G$ be a group acting geometrically on a CAT(0) cube complex $X$. If $X$ contains two transverse trees of hyperplanes that both are branching, uniformly distributed, and have no transeparation, then $G$ is not cyclically hyperbolic.
\end{prop}

\noindent
Before turning to the proof of the proposition, we prove two elementary observations:

\begin{lemma}\label{lem:TranslatesDisjoint}
Let $X$ be a finite-dimensional CAT(0) cube complex and $g \in \mathrm{Isom}(X)$ an isometry admitting an axis $\gamma$. For every hyperplane $J$ crossing $\gamma$, the hyperplanes in $\{ g^{k \cdot \dim(X)!} J \mid k \in \mathbb{Z}\}$ are pairwise disjoint.
\end{lemma}

\begin{proof}
Because $X$ does contain $\dim(X)+1$ pairwise transverse hyperplanes, at least two hyperplanes in $\{g^k J \mid 0 \leq k \leq \dim(X)\}$ are not transverse. So there exists $1 \leq p \leq \dim(X)$ such that $J$ and $g^pJ$ are not transverse, which implies that $\{g^{kp}J \mid k \in \mathbb{Z}\}$ is a collection of pairwise non-transverse hyperplanes crossing $\gamma$. Because $\{g^{k \cdot \dim(X)!}J \mid k \in \mathbb{Z} \}$ is a subcollection, the desired conclusion follows. 
\end{proof}

\begin{lemma}\label{lem:Trans}
Let $X$ be a finite-dimensional CAT(0) cube complex and $g \in \mathrm{Isom}(X)$ an isometry admitting an axis $\gamma$ (endowed with a left-right order  once identified as a copy of $\mathbb{R}$  such that $g$ translates points to the right). There exists a constant $C \geq 0$ such that, for every hyperplane $A$ crossing $\gamma$ and every hyperplane $B$ crossing $\gamma$ on the right of $A$, if $A$ is transverse to at least $C$ of the $\langle g^{\dim(X)!} \rangle^+$-translates of $B$, then $A$ is transverse to all the $\langle g^{\dim(X)!} \rangle^+$-translates of $B$
\end{lemma}

\noindent
 Here, $\langle \cdot \rangle^+$ denote the subsemigroup generated by the subset under consideration. 

\begin{proof}[Proof of Lemma~\ref{lem:Trans}.]
Given a hyperplane $J$ crossing $\gamma$, let $r(J)$ denote the set of hyperplanes crossing $\gamma$ on the right of $J$ at distance $\leq \|g\| \cdot \dim(X)!$. For every $H \in r(J)$, let $\rho(J,H)$ denote the number of positive $\langle g^{\dim(X)!} \rangle$-translates of $H$ transverse to $J$. Set 
$$\rho(J):= \max \{ \rho(J,H) \mid H \in r(J) \text{ and } \rho(J,H)<\infty\}.$$
Observe that $\rho(gJ)=\rho(J)$ for every $g \in G$, so, because there exist only infinitely many $\langle g \rangle$-orbits of hyperplanes crossing $\gamma$, $C:= \max \{ \rho(J) \mid J \text{ hyperplane crossing } \gamma\}$ is finite. 

\medskip \noindent
Now, let $A,B$ be two hyperplanes crossing $\gamma$ such that $B$ crosses $\gamma$ on the right of $A$ and such that $A$ is transverse to more than $C$ positive $\langle g^{\dim(X)!} \rangle$-translates of $B$. Let $B'$ the leftmost negative $\langle g^{\dim(X)!} \rangle$-translate of $B$ that still crosses $\gamma$ on the right of $A$. Of course, the positive $\langle g^{\dim(X)!} \rangle$-translates of $B'$ are positive $\langle g^{\dim(X)!} \rangle$-translates of $B$, so $A$ is transverse to $>C \geq \rho(A)\geq \rho(A,B')$ positive $\langle g^{\dim(X)!} \rangle$-translates of $B'$, which implies that $A$ is actually transverse to all the positive $\langle g^{\dim(X)!} \rangle$-translates of $B'$ (and a fortiori of $B$).
\end{proof}

\begin{proof}[Proof of Proposition~\ref{prop:NotHyp}.]
Given elements $g_1, \ldots, g_k \in G$, we want to prove that $G$ is not weakly hyperbolic relative to $\{ \langle g_1 \rangle, \ldots, \langle g_k \rangle\}$. As a consequence of Lemma~\ref{lem:MilnorSvarc}, this amounts to fixing an axis $\gamma_i$ of $g_i$ for every $1 \leq i \leq k$ (such axes existing up to subdividing $X$ if necessary) and to showing that the cone-off $\dot{X}$ of $X$ over the $G$-translates of the $\gamma_i$ is not hyperbolic. More precisely, we are going to construct a collection of larger and larger flat rectangles in $X$ that uniformly quasi-isometrically embed in $\dot{X}$. 

\medskip \noindent
For this purpose, up to replacing the $g_i$ with their $(\dim(X)!)$th powers, we assume that the conclusions of Lemmas~\ref{lem:TranslatesDisjoint} and~\ref{lem:Trans} hold for each $g_i$; let $C$ denote the maximum of the constants given by Lemma~\ref{lem:Trans} for the $g_i$. Also, we let $\mathcal{H},\mathcal{V}$ denote the two underlying collections of hyperplanes of our transverse trees of hyperplanes, and we fix a sufficiently large constant $L$ such that
\begin{itemize}
	\item $L \geq (\tau +2C-1)\tau +2$
	\item and $L^2 > KL \tau +2$,
\end{itemize}
where $K$ denotes the maximum of the constants given by the uniform distribution of our trees of hyperplanes and where $\tau$ denotes the maximum of the translation lengths of the $g_i$. Our first step is to construct \emph{rays of hyperplanes} in $\mathcal{H},\mathcal{V}$, i.e.\ sequences of hyperplanes in which each hyperplane is a children of the previous one, that do not ``interact'' too much with the translates of the $\gamma_i$. More precisely:

\begin{claim}\label{claim:Rays}
Set $\mathcal{W}=\mathcal{H}$ or $\mathcal{V}$. There exists a ray of hyperplanes $r_\mathcal{W}$ in $\mathcal{W}$ such that no $G$-translate of a $\gamma_i$ crosses $L(L+1)$ consecutive hyperplanes of $r_\mathcal{W}$. 
\end{claim}

\noindent
We postpone the proof of this claim and show first how to conclude from here. Let $H_1,H_2,\ldots \subset \mathcal{H}$ and $V_1, V_2, \ldots \subset \mathcal{V}$ denote the two rays of hyperplanes given by Claim~\ref{claim:Rays}. For every $i \geq 2$, let $R_i$ denote the flat rectangle given by Proposition~\ref{prop:FlatRectangles} from the cycle of convex subcomplexes $(N(H_1),N(V_i),N(H_i),N(V_1))$; and $R_i^-$ the subrectangle of $R_i$ delimited by $V_2, V_{i-1}, H_2,H_{i-1}$. Fix an $i \geq 4$ and two vertices $x,y \in R_i^-$. Let $n,m,p,q$ denote the indices such that the hyperplanes in our rays separating $x$ and $y$ are precisely $H_n, \ldots, H_m$ and $V_p, \ldots, V_q$. If $\xi$ is a geodesic in the cone-off $\dot{X}$ between $x$ and $y$, it follows from Claim~\ref{claim:Rays} that $\xi$ has a vertex between $V_{n+kL(L+1)}$ and $V_{n+L(k+1)(L+1)}$ for every $0 \leq k < (m-n)/L(L+1)$, so the distance between $x$ and $y$ in $\dot{X}$ is at least $(m-n)/L(L+1)$. Similarly, one shows that the distance between $x$ and $y$ in $\dot{X}$ is at least $(q-p)/L(L+1)$. Now, we know from Proposition~\ref{prop:FlatRectangles} that the hyperplanes crossing $R_i$ are precisely the hyperplanes separating $V_1,V_i$ or $H_1,H_i$. Because our trees of hyperplanes have no transeparation, we deduce that, for every $2 \leq j \leq i-2$, the hyperplanes separating $V_j$ and $V_{j+1}$ in $R_i$ separates $V_{j-1}$ and $V_{j+2}$ in $X$, hence 
$$d_{R_i} (V_j,V_{j+1}) \leq d_X(V_{j-1},V_{j+2}) \leq 3K.$$
So the distance between $x$ and $y$ in $R_i$ is at most $3K(m-n+q-p+4)$. Thus we have proved that
$$\frac{1}{6KL(L+1)} d_{R_i}(x,y)-\frac{2}{L(L+1)} \leq d_{\dot{X}}(x,y) \leq d_{R_i}(x,y).$$
Thus, the flat rectangles $R_i^-$ are uniformly quasi-isometrically embedded in the cone-off $\dot{X}$. Because each $R_i^-$ is isometric to a $[0,a_i]\times[0,b_i]$ for some $a_i,b_i \geq i-4$, we conclude that $\dot{X}$ is not hyperbolic, as desired.

\medskip \noindent
From now on, we focus on the proof of Claim~\ref{claim:Rays}. Given an arbitrary ray of hyperplanes $r$, one says that a hyperplane in $r$ is \emph{bad} if there exists a $G$-translate of some $\gamma_i$ that crosses it together with the $L^2+L-1$ following hyperplanes; and we denote by $\mathrm{bad}(r)$ the index of the first bad hyperplane along $r$. 

\medskip \noindent
Now, fix an arbitrary ray of hyperplanes $r_0 \subset \mathcal{W}$. If $r_0$ has no bad hyperplane, then it is the ray we are looking for. Otherwise, let $J_1$ denote the first bad hyperplane of $r_0$, and let $J_2,\ldots, J_{L^2+L}$ denote the $L^2+L-1$ next hyperplanes. By definition of being bad, there exists a $G$-translate $\alpha$ of some $\gamma_i$ such that $\alpha$ crosses $J_1, \ldots, J_{L^2+L}$; let $a$ denote the corresponding $G$-conjugate of $g_i$ (up to replacing $a$ with $a^{-1}$, we assume that the positive orientation of $\alpha$ goes from $J_1$ to $J_2$). Because our tree of hyperplanes is branching, $J_{L^2}$ admits a children $A \in \mathcal{W}$ distinct from $J_{L^2+1}$. We construct a new ray of hyperplanes $r_1$ by cutting $r_0$ after $J_{L^2}$, replacing $J_{L^2+1}$ with $A$, and extending it in an arbitrary way.

\medskip \noindent
By construction, $r_0$ and $r_1$ have the same $(\mathrm{bad}(r_0)+L^2-1)$ first hyperplanes. We claim that $\mathrm{bad}(r_1)>\mathrm{bad}(r_0)$. This will allow us to conclude the proof of Claim~\ref{claim:Rays} because, by iterating the construction, one gets a sequence of rays of hyperplanes $r_0,r_1,\ldots$ either that stops at some $r_k$ without bad hyperplanes, producing the desired ray of hyperplanes, or that goes for ever with the property that $(\mathrm{bad}(r_i))$ increases and that $r_i,r_{i+1}$ agree on the first $\mathrm{bad}(r_i)$ hyperplanes for every $i \geq 0$. In the latter case, the sequence $(r_i)$ converges to a ray of hyperplanes with no bad hyperplanes, as desired.

\medskip \noindent
Assume for contradiction that $\mathrm{bad}(r_1)\leq \mathrm{bad}(r_0)$. As a consequence, there exists a $G$-translate $\beta$ of some $\gamma_i$ that crosses $J_1,\ldots, J_{L^2},A$; let $b$ denote the corresponding $G$-conjugate of $g_i$. The configuration is summarised by Figure~\ref{RayHyp}.

\medskip \noindent
Because $L \geq (\tau +2C-1)\tau +2$, $\{J_{L^2+2}, \ldots, J_{L^2+L}\}$ contains a subcollection $\mathcal{J}$ of size $\geq \tau +2C$ that lies in a single $\langle a \rangle$-orbit. 

\begin{figure}
\begin{center}
\includegraphics[width=0.7\linewidth]{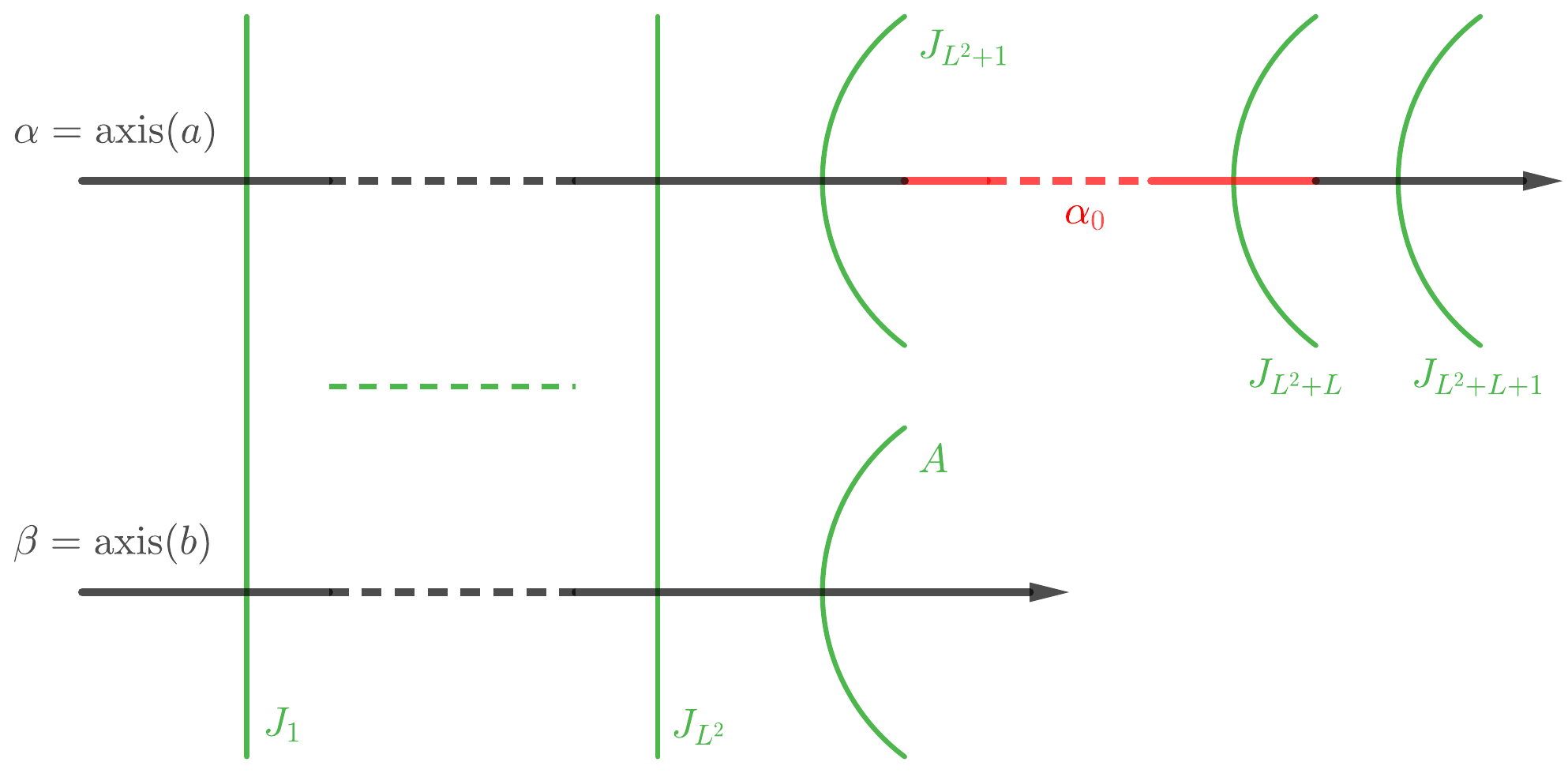}
\caption{How to construct a ray of hyperplanes with no bad hyperplane.}
\label{RayHyp}
\end{center}
\end{figure}
\begin{fact}\label{fact:CrossingBeta}
There exists an $\langle a \rangle$-translate of $\mathcal{J}$ that contains at least $\tau+1$ hyperplanes crossing~$\beta$.
\end{fact}

\noindent
Assume for now that it is not the case. Let $P,Q$ denote respectively the first and last hyperplanes in $\mathcal{J}$ along $r_0$. We know that there exists some $q \geq 0$ such that $Q=a^qP$. Because $a$ always sends a hyperplane crossing its axis $\alpha$ to a non-transverse hyperplane, it follows that $P,aP, \ldots, a^qP$ separate $J_{L^2+1}$ and $J_{L^2+L+1}$, hence $p+1 \leq d(J_{L^2+1},J_{L^2+L+1}) \leq KL$. As a consequence, if $\alpha_0$ denotes the shortest subsegment of $\alpha$ crossed by $P,Q$, then its length is at most $(p+1) \| a \| \leq KL \tau$. On the other hand, the subsegment of $\alpha$ delimited by $J_1$ and $J_{L^2}$ has length at least
$$d(J_1,J_{L^2}) \geq \# \{J_2, \ldots, J_{L^2-1} \} = L^2 -2> LK \tau \geq \ell.$$
Therefore, there exists some negative $p \in \mathbb{Z}$ such that all the hyperplanes in $a^p \cdot \mathcal{J}$ intersect $\alpha$ between $J_1$ and $J_{L^2}$. By assumption, there exist at most $\tau$ hyperplanes in $a^p \cdot \mathcal{J}$ crossing $\beta$, so at least $(|\mathcal{J}|-\tau)/2 \geq C$ hyperplanes in $a^p \cdot \mathcal{J}$ crosses either $J_1$ or $J_{L^2}$. We distinguish two cases.

\medskip \noindent
\underline{Case 1:} there exist at least $C$ hyperplanes in $a^p \cdot \mathcal{J}$ crossing $J_{1}$. By applying Lemma~\ref{lem:Trans} to these hyperplanes, $a$, and $J_1$, we conclude that $\mathcal{J}$ contains a hyperplane transverse to $J_1$ (namely, a positive $\langle a \rangle$-translate of a hyperplane in $a^p \cdot \mathcal{J}$ transverse to $J_1$). But this is impossible since, in the ray of hyperplanes $\mathcal{W}$, $J_{L^2}$ separates $\mathcal{J}$ from $J_1$.

\medskip \noindent
\underline{Case 2:} there exist at least $C$ hyperplanes in $a^p \cdot \mathcal{J}$ crossing $J_{L^2}$. By applying Lemma~\ref{lem:Trans} to these hyperplanes, $a^{-1}$, and $J_{L^2}$, we conclude that there exist infinitely many negative $\langle a \rangle$-translate of hyperplanes in $a^p \cdot \mathcal{J}$ that cross $J_{L^2}$. Among such a collection, only finitely many of them may cross $\beta$, by assumption. So one gets infinitely many hyperplanes separating $\alpha \cap J_{L^2}$ and $\beta \cap J_{L^2}$, which is impossible.

\medskip \noindent
Thus, we have proved Fact~\ref{fact:CrossingBeta}. Let $\mathcal{J}'$ denote an $\langle a \rangle$-translate of $\mathcal{J}$ that contains at least $\tau +1$ hyperplanes crossing $\beta$. We already know that $\mathcal{J}'$ lies in a single $\langle a \rangle$-orbit. On the other hand, $\mathcal{J}'$ necessarily contains two hyperplanes in the same $\langle b \rangle$-orbit. As a consequence, there exist two hyperplanes $H,H'$ crossing $\alpha,\beta$ such that $H'$ lies both in the $\langle a \rangle$- and $\langle b \rangle$-orbits of $H$. As a consequence, $H$ has infinitely many $\langle a \rangle$-translates crossing both $\alpha,\beta$, which implies that all the $\langle a \rangle$-translates of $H$ cross both $\alpha,\beta$ (since $a$ sends a hyperplane crossing $\alpha$ to a disjoint hyperplane). But $H$ has an $\langle a \rangle$-translate in $\mathcal{J}$, and no hyperplane in $\mathcal{J}$ can cross $\beta$ since it is separated from $\beta$ by $J_{L^2+1}$. This concludes the proof of Claim~\ref{claim:Rays}.
\end{proof}

\subsection{Constructing trees of hyperplanes}

\noindent
Of course, in order to apply Proposition~\ref{prop:NotHyp}, we have to be able to construct trees of hyperplanes. Our main result in this direction is the following:

\begin{lemma}\label{lem:TreeHyp}
Let $X$ be a CAT(0) cube complex and $a,b \in \mathrm{Isom}(X)$ two isometries. Assume that there exist hyperplanes $A,B,J$ satisfying $aJ^+ \subset A^+$ and $bJ^+ \subset B^+$, where $A^+,B^+,J^+$ are halfspaces respectively delimited by $A,B,J$ that satisfy $A^+,B^+ \subset J^+$ and $A^+ \cap B^+ = \emptyset$. Then $\{ w J \mid w \in \langle a,b \rangle^+ \}$ defines a branching and uniformly distributed tree of hyperplanes.  Moreover, if $A,B,J$ are pairwise strongly separated, then any two hyperplanes in our tree are strongly separated. 
\end{lemma}
\begin{figure}
\begin{center}
\includegraphics[width=0.7\linewidth]{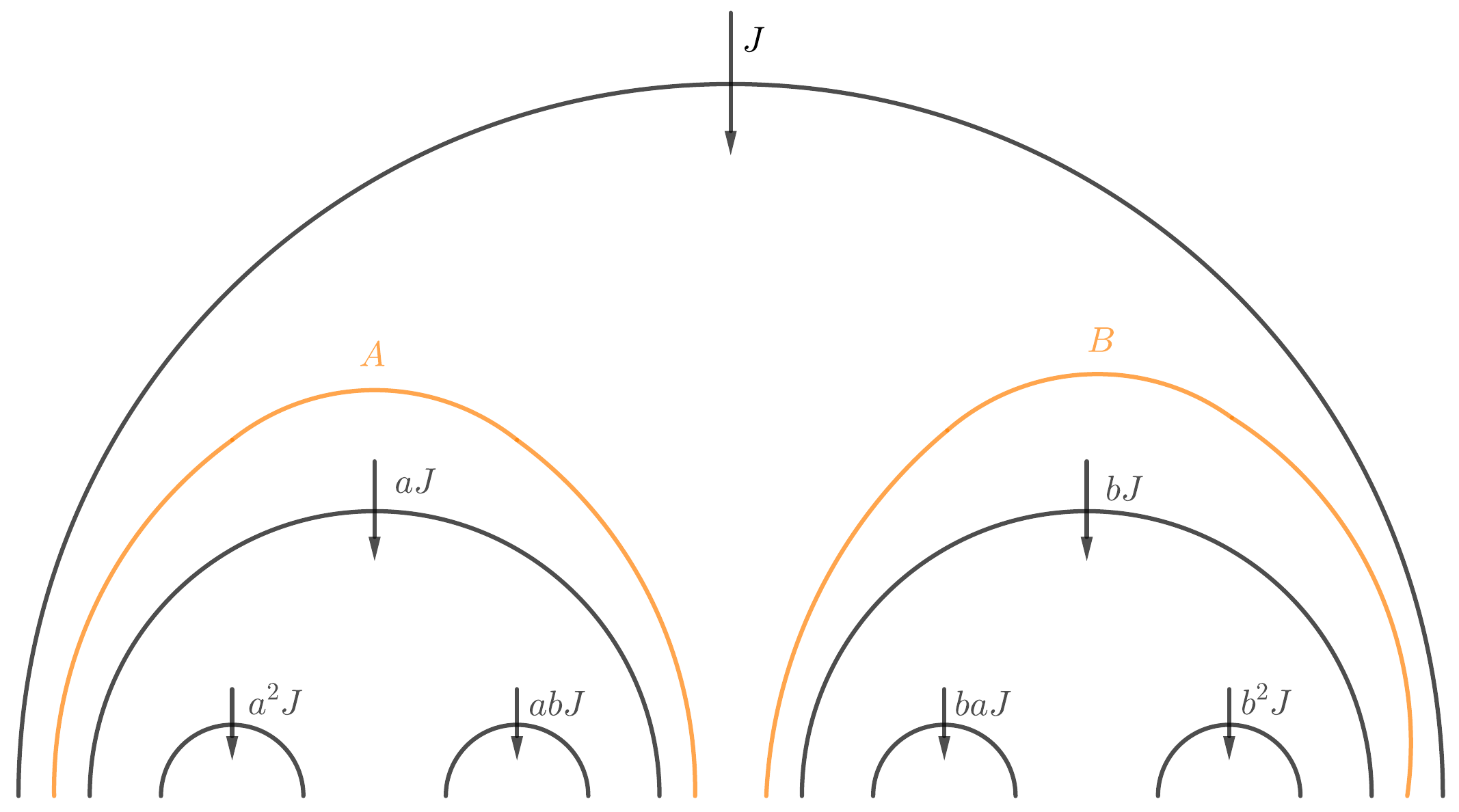}
\caption{Tree of hyperplanes constructed by Lemma~\ref{lem:TreeHyp}.}
\label{TreeHyp}
\end{center}
\end{figure}

\noindent
 Recall that we denote by $\langle \cdot \rangle^+$ the subsemigroup generated by the subset under consideration. 

\begin{proof}[Proof of Lemma~\ref{lem:TreeHyp}.]
 By a classical ping-pong argument, it is clear that $\langle a,b \rangle^+$ is free with $\{a,b\}$ as a basis. Therefore, we can think of $\langle a,b \rangle^+$ as the set of words written over $\{a,b\}$.  Let $T$ denote the rooted tree whose vertex-set is $\langle a,b\rangle^+$, whose root is the empty word, and whose edges connect two words if one can be obtained from the other by removing the last letter. For every vertex $u \in T$, set $H_u:=uJ$. See Figure~\ref{TreeHyp}. Observe that:

\begin{claim}\label{claim:One}
For all $u,v \in \langle a,b\rangle^+$, if $u$ is a prefix of $v$ then $uJ^+ \supset vJ^+$. 
\end{claim}

\noindent
It amounts to saying that $uJ^+ \subset J^+$ for every $u \in \langle a,b\rangle^+$, which can be proved by induction. Indeed, if $a$ is the last letter of $u$, then we can write $u$ as $wa$ and
$$uJ^+ = waJ^+ \subset wA^+ \subset wJ^+ \subset J^+.$$
One argues similarly if $b$ is the last letter of $u$. Next, observe that:

\begin{claim}\label{claim:Two}
For all $u,v \in \langle a,b\rangle^+$, if no one is a prefix of the other then $uJ^+ \cap vJ^+ =\emptyset$.
\end{claim}

\noindent
Without loss of generality, assume that $u,v$ have no common prefix. Up to switching $u$ and $v$, write $u$ as $au_0$ and $v$ as $bv_0$. Then, by applying Claim~\ref{claim:One}, we get
$$u J^+ \cap vJ^+ = au_0 J^+ \cap bv_0J^+ \subset aJ^+ \cap bJ^+ \subset A^+ \cap B^+ = \emptyset,$$
proving our claim.

\medskip \noindent
The combination of Claims~\ref{claim:One} and~\ref{claim:Two} shows that, for all vertices $u,v \in T$, the hyperplane $H_u$ separates $H_\emptyset=J$ and $H_v$ if and only if $u$ is a prefix of $v$, or equivalently that $u$ belongs to the geodesics from $\emptyset$ to $v$. This implies that our collection of hyperplanes defines a tree of spaces, which is branching by construction. In order to show that it is also uniformly distributed, fix two vertices $u,v \in T$ and write $u=wu_0$, $v=wv_0$ where $w$ is the maximal common prefix of $u,v$. Then
$$\begin{array}{lcl} d(H_u,H_v) & = & d(uJ,vJ) = d(u_0J,v_0J) = d(J,u_0^{-1}v_0 J) \\ \\ & \leq & K \cdot |u_0^{-1}v_0| \leq K \cdot \left( |u_0|+|v_0| \right) = K \cdot d_T(u,v) \end{array}$$
where $K:= \max( d(J,aJ),d(J,bJ))$. This completes the proof of the first assertion of our lemma.

\medskip \noindent
 Finally, assume that $A,B,J$ are pairwise strongly separated. Let $H_u$ and $H_v$ be two hyperplanes from our tree. As before, write $u=wu_0$ and $v = wv_0$ where $w$ is the maximal common prefix of $u$ and $v$. The first letter of $u_0$ must be different from the first letter of $v_0$, say that $u_0=au_1$ and $v_0=bu_1$ up to switching $a$ and $b$. According to Claim~\ref{claim:One}, $H_u \subset waJ^+$ and $H_v \subset wbJ^+$. Thus, in order to conclude that $H_u$ and $H_v$ are strongly separated, it suffices to verify that $waJ$ and $wbJ$, or equivalently $aJ$ and $bJ$, are strongly separated. But, on the one hand, $aJ \subset A^+$ and $bJ \subset B^+$; and, on the other hand, $A$ and $B$ are strongly separated by assumption. So the desired conclusion follows. 
\end{proof}

\section{Cyclically hyperbolic cubulable groups}

\subsection{Preliminaries}

\noindent
In this section, we record a few standard results and definitions related to CAT(0) cube complexes. 

\paragraph{Roller boundary.} Given a CAT(0) cube complex $X$, an \emph{orientation} of $X$ is a map $\sigma$ that associates to every hyperplane of $X$ one of the two halfspaces it delimits in such that $\sigma(J_1) \cap \sigma(J_2) \neq \emptyset$ for all hyperplanes $J_1,J_2$ of $X$. For instance, given a vertex $x \in X$, the map $\sigma_x$ that sends each hyperplane to its halfspace containing $x$ is an orientation, referred to as a \emph{principal orientation}. The \emph{Roller completion} $\overline{X}$ of $X$ is the cube-completion of the graph whose vertices are the orientations of $X$ and whose edges connect two orientations whenever they differ on a single hyperplane. Each connected component of $\overline{X}$ is a CAT(0) cube complex, and the map $x \mapsto \sigma_x$ induces an isomorphism from $X$ to the component of $\overline{X}$ containing the principal orientations. One identifies $X$ with its image in $\overline{X}$ and one refers to $\mathfrak{R}X:= \overline{X} \backslash X$ as the \emph{Roller boundary} of $X$. 

\medskip \noindent
Given a component $Y \subset \mathfrak{R}X$, halfspaces and hyperplanes of $Y$ can be identified with halfspaces and hyperplanes of $X$. More precisely, for every hyperplane $J$ separating at least two vertices in $Y$ (i.e.\ there exist at least two vertices of $Y$ that are orientations of $X$ differing at $J$) and for every halfspace $J^+$ delimited by $J$, the subcomplex of $Y$ given by the orientations $\sigma$ of $X$ satisfying $\sigma(J)=J^+$ is a halfspace of $Y$. Conversely, every halfspace of $Y$ arises in this way. As claimed, this observation allows us to identify halfspaces and hyperplanes of $Y$ with halfspaces and hyperplanes of $X$. Moreover, this identification preserves respectively intersection and transversality. 

\medskip \noindent
An orientation $\sigma$ of our component $Y$ uniquely extends to an orientation $\sigma^+$ of $X$: if a hyperplane $J$ crosses $Y$, set $\sigma^+(J)$ as the halfspace of $X$ corresponding the halfspace $\sigma(J)$ of $Y$; and otherwise set $\sigma^+(J)$ as the halfspace delimited by $J$ containing $Y$. Thus, the Roller completion of $\overline{Y}$, where $Y$ is thought of as a CAT(0) cube complex on its own right, naturally embeds into $\overline{X}$. In the sequel, we identify $\overline{Y}$ with its image in $\overline{X}$.

\medskip \noindent
Finally, we record a couple of statements for future use.

\begin{lemma}\label{lem:ProjBoundary}\emph{(\cite[Proposition~2.9]{MR4283595})}
Let $X$ be a finite-dimensional CAT(0) cube complex, $x \in X$ a vertex, and $Y \subset \mathfrak{R}X$ a component. There exists a unique vertex $\xi \in Y$ such that the hyperplanes separating $x$ from $\xi$ are precisely the hyperplanes separating $x$ from $Y$. 
\end{lemma}

\noindent
 In the sequel, we will refer to the vertex $\xi$ as the \emph{projection of $x$ to $Y$}. 

\begin{lemma}\label{lem:RollerDim}
Let $X$ be a finite-dimensional CAT(0) cube complex. If $Y \subset \mathfrak{R}X$ is a component, then $\dim(Y)<\dim(X)$.
\end{lemma}

\begin{proof}
If $\dim(Y) \geq n$, then $Y$ contains $n$ pairwise transverse hyperplanes. These hyperplanes correspond to $n$ pairwise transverse hyperplanes $J_1, \ldots, J_n$ in $X$, and they have to cross in an $n$-cube $C \subset X$. Fix an arbitrary vertex $x \in X$ and let $\xi \in Y$ be the projection given by Lemma~\ref{lem:ProjBoundary}. Because there exist infinitely many hyperplanes separating $x,\xi$ and because $X$ is finite-dimensional, there must exist a decreasing sequence of halfspaces $D_1 \supset D_2 \supset \cdots$ such that $Y \subset D_i$ but $x\notin D_i$ for every $i \geq 1$.  Because $\bigcap_{i \geq 1} D_i = \emptyset$ in $X$ , we can fix some $i \geq 1$ large enough such that $C \cap D_i = \emptyset$.  On the one hand, $J_1, \ldots, J_n$ all cross $C$ and $Y$; on the other hand, the hyperplane $K_i$ delimiting $D_i$ separates $C$ and $Y$ . Therefore, $K_i,J_1, \ldots, J_n$ must be pairwise transverse, hence $n +1 \leq \dim(X)$. 
\end{proof}

\paragraph{Contracting isometries.} In a metric space $X$, an isometry $g$ is \emph{contracting} if there exist $x \in X$ and $D \geq 0$ such that $n \mapsto g^nx$ induces a quasi-isometric embedding $\mathbb{N} \to X$ and such that the orbit $\langle g \rangle \cdot x$ is \emph{$D$-contracting} in the sense that nearest point projection on $\langle g \rangle \cdot x$ of every ball disjoint from $\langle g \rangle \cdot x$ has diameter $\leq D$. When $X$ is a CAT(0) cube complex, contracting isometries (with respect to the combinatorial metric) can be characterised as follows:

\begin{prop}\label{prop:ContractingCharacterisation}\emph{(\cite[Theorem~3.13]{article3})}
Let $X$ be a CAT(0) cube complex. A loxodromic isometry $g \in \mathrm{Isom}(X)$ is contracting if and only if there exist a constant $L \geq 0$ and a hyperplane $J$ such that $\{g^n \cdot J \mid n \in \mathbb{Z}\}$ is a collection of pairwise $L$-well-separated hyperplanes.
\end{prop}

\noindent
Recall that, in a CAT(0) cube complex, a \emph{facing triple} is the data of three hyperplanes such that no one separates the other two. Two hyperplanes $J_1,J_2$ are \emph{$L$-well-separated} if they are not transverse and if every collection of hyperplanes transverse to both $J_1,J_2$ with no facing triple has cardinality $\leq L$. If $L=0$, i.e.\ if there is no hyperplane transverse to both $J_1$ and $J_2$, one says that $J_1$ and $J_2$ are \emph{strongly separated}. 

\begin{lemma}\label{lem:ContractingQuasiconvex}\emph{(\cite[Lemmas~2.20 and~2.21]{article3})}
Let $X$ be a CAT(0) cube complex and $g\in \mathrm{Isom}(X)$ a loxodromic isometry. If $g$ is contracting, then, for every axis $\gamma$ of $g$, the Hausdorff distance between $\gamma$ and its convex hull is finite.
\end{lemma}

\noindent
In the next statement, we refer to two contracting isometries as \emph{independent} if the intersection between any two neighbourhoods of any two of their axes is always bounded.

\begin{lemma}\label{lem:ContractingInfinity}
Let $X$ be a finite-dimensional CAT(0) cube complex and $g_1,g_2 \in \mathrm{Isom}(X)$ two loxodromic isometries. If $g_1,g_2$ are contracting and independent, then the points at infinity of axes of $g_1,g_2$ belong to distinct components of $\overline{X}$.
\end{lemma}

\begin{proof}
It suffices to show that, if $g \in \mathrm{Isom}(X)$ is a contracting isometry admitting a  combinatorial  axis $\gamma$, then the component $Y$ of $\mathfrak{R}X$ containing $\gamma(+ \infty)$ is bounded. As a consequence of Proposition~\ref{prop:ContractingCharacterisation}, there exist hyperplanes $\ldots, J_{-1},J_0,J_1,\ldots$ crossing $\gamma$ in that order that are pairwise $L$-well-separated for some $L \geq 0$. If $Y$ is unbounded, then there exist a decreasing sequence $D_1 \supset \cdots \supset D_{L+1}$ of halfspaces in $Y$ (which we also identify with a decreasing sequence of halfspaces in $X$). Fix two vertices $x,y \in X$ such that $x \notin D_1$ and $y \in D_{L+1}$. There exists some $k \in \mathbb{Z}$ such that $J_i$ separates $x,y$ from $Y$ for every $i \geq k$. Consequently, the hyperplanes delimiting $D_1, \ldots, D_{L+1}$ are all transverse to both $J_k$ and $J_{k+1}$, contradicting the fac that $J_k$ and $J_{k+1}$ are $L$-well-separated. We conclude that $Y$ has to be bounded.
\end{proof}

\noindent
 Despite the fact that our next statement holds for more general metric spaces, we restrict ourself to CAT(0) cube complexes in order to shorten the argument.

\begin{lemma}\label{lem:PurelyContractingFree}
Let $X$ be a CAT(0) cube complex. If $a,b \in \mathrm{Isom}(X)$ are two independent contracting isometries, then there exists $k \geq 1$ such that $\langle a^k,b^k \rangle$ is free of rank two and \emph{purely contracting} (i.e.\ every non-trivialement element of $\langle a^k,b^k \rangle$ is a contracting isometry). Moreover, for all non-trivial $g,h \in \langle a^k,b^k\rangle$, if $\langle g \rangle \cap \langle h \rangle= \{1\}$ then $g$ and $h$ are independent isometries.
\end{lemma}

\begin{proof}
According to Proposition~\ref{prop:ContractingCharacterisation}, there exist a constant $L \geq 0$ and two hyperplanes $J,H$ such that $\{a^n J \mid n \in \mathbb{Z} \}$ and $\{b^n J \mid n \in \mathbb{Z}\}$ are two collections of $L$-well-separated hyperplanes. Fix two axes $\alpha,\beta$ of $a,b$ respectively. Observe that:

\begin{claim}\label{claim:FewCommonHyp}
Only finitely many $\langle a \rangle$-translates of $J$ crosses $\beta$. Similarly, only finitely many $\langle b \rangle$-translates of $H$ crosses $\alpha$.
\end{claim}

\noindent
The two assertions are symmetric, so we only prove the first one. Because $a$ and $b$ are independent, it suffices to show that, if $a^iJ$ and $a^{i+1}J$ both cross $\beta$ for some $i \in \mathbb{Z}$, then $d(x,y) \leq 2L + 3 \|a\|$ for all $x \in\alpha$ and $y \in \beta$ lying between $a^iJ$ and $a^{i+1}J$. Indeed, because $a^iJ$ and $a^{i-1}J$ are $L$-well-separated, at most $L$ hyperplanes separating $x$ and $y$ are transverse to $a^{i-1}J$. Similarly, at most $L$ hyperplanes separating $x$ and $y$ are transverse to $a^{i+2}J$. Therefore, among the hyperplanes separating $x$ and $y$, all but at most $2L$ intersect the subsegment of $\alpha$ delimited by $a^{i-1}J$ and $a^{i+2}J$. Since this subsegment has length $3 \|a\|$, we conclude that $d(x,y) \leq 2L+3\|a\|$, as desired.

\medskip \noindent
As a consequence of Claim~\ref{claim:FewCommonHyp}, up to replacing $J$ with an $\langle a \rangle$-translate and $a$ with one of its powers, we can assume that $J$, $aJ$, $H$, and $bH$ define a facing quadruple, i.e.\ none of these hyperplanes separates two of the others. Let $J^+$ denote the halfspace delimited by $J$ that contains $H$ and $bH$, and $H^+$ the halfspace delimited by $H$ that contains $J$ and $aJ$. For all $D \geq 0$, we have
$$a^D(aJ^+ \cup H^- \cup bH^+) \subset a^D J^+ \text{ and } a^{-D}(J^- \cup H^- \cup bH^+) \subset a^{-D-1}J^-,$$
and similarly
$$b^D(bH^+ \cup J^- \cup aJ^+) \subset b^D H^+ \text{ and } b^{-D}( H^- \cup J^- \cup aJ^+) \subset b^{-D-1}H^-.$$
We are in good position to play ping pong, but we have to use a different metric. For all vertices $x,y \in X$, set
$$\delta_L(x,y) := \text{maximal number of pairwise $L$-well-separated hyperplanes separating $x$ and $y$}.$$
In \cite[Section~6.6]{MR4057355}, we proved that $\delta_L$ defines a new quasi-geodesic metric on (the vertices of) $X$ such that every isometry of $X$ with unbounded orbits in $(X,\delta_L)$ must be contracting. With respect to the metric $\delta_L$, the inclusions above imply that, for every $D_0 \leq D-1$, we have
$$a^D(aJ^+ \cup H^- \cup bH^+)^{+D_0} \subset a J^+ \text{ and } a^{-D}(J^- \cup H^- \cup bH^+)^{+D_0} \subset J^-,$$
and similarly
$$b^D(bH^+ \cup J^- \cup aJ^+)^{+D_0} \subset b H^+ \text{ and } b^{-D}( H^- \cup J^- \cup aJ^+)^{+D_0} \subset H^-.$$
Then, choosing $D_0$ sufficiently large so that $(X,\delta_L)$ is $D_0$-quasi-geodesic, the quasi-isometric ping-pong lemma \cite[Lemma~2.4]{MR2390351} applies and shows that $\langle a^D,b^D \rangle$ is a free group of rank two and that its orbits are quasi-isometrically embedded in $(X,\delta_L)$. 

\medskip \noindent
Consequently, every non-trivial element of $\langle a^D,b^D \rangle$ has unbounded orbits with respect to $\delta_L$, and consequently must be contracting. And, for all non-trivial elements $g,h \in \langle a^D,b^D \rangle$, if the intersection between two neighbourhoods of two axes of $g$ and $h$ is unbounded in $X$, then it must be unbounded in $(X,\delta_L)$, so $g$ and $h$ must have equal non-trivial powers. 
\end{proof}

\paragraph{Caprace-Sageev machinery.} Finally, we record for future use a couple of consequences of results from \cite{MR2827012}.

\begin{prop}\label{prop:Essential}
Let $G$ be a group acting on a finite-dimensional CAT(0) cube complex $X$. Assume that one of the following conditions is satisfied:
\begin{itemize}
	\item $X$ contains only finitely many $G$-orbits of hyperplanes;
	\item $G$ does not stabilise a component in $\mathfrak{R}X$.
\end{itemize}
Then there exists a convex subcomplex $Y \subset X$ on which $G$ acts essentially.
\end{prop}

\noindent
Recall that an action on a CAT(0) cube complex is \emph{essential} if no orbit lies in a neighbourhood of some halfspace.

\begin{proof}[Proof of Proposition~\ref{prop:Essential}.]
Our proposition is a direct consequence of \cite[Proposition~3.5]{MR2827012} and the following observation:

\begin{claim}\label{claim:VisualBoundary}
If $G$ fixes a point in the visual boundary of $X$, then $G$ stabilises a component in $\mathfrak{R}X$.
\end{claim}

\noindent
Let $r$ be a CAT(0) ray representing a point in the visual boundary fixed by $G$. The map $\sigma_r$ that sends each hyperplane to its halfspace containing eventually $r$ defines an orientation. Given a $g \in G$, we know that the Hausdorff distance with respect to the CAT(0) metric between $r$ and $gr$ is finite, so there may exist only finitely many hyperplanes on which $\sigma_r$ and $\sigma_{gr}$ differ. In other words, $\sigma_r$ and $\sigma_{gr}$ belongs to the same component. We conclude that the component of $\mathfrak{R}X$ containing $\sigma_r$ is stabilised by $G$. 
\end{proof}

\begin{prop}\label{prop:ForCriterion}
Let $G$ be a group acting essentially on an irreducible CAT(0) cube complex $X$ of finite dimension. Assume that $\overline{X}$ has infinitely many components and that
\begin{itemize}
	\item either $X$ is locally finite and $G$ acts cocompactly;
	\item or $G$ does not stabilise a component of $\mathfrak{R}X$.
\end{itemize}
Then there exist two elements $a,b \in G$ and a facing triple of three pairwise strongly separated hyperplanes $A,B,J$ such that $aJ^+ \subset A^+$ and $bJ^+ \subset B^+$ for some halfspaces $A^+,B^+,J^+$ respectively delimited by $A,B,J$ satisfying $A^+,B^+ \subset J^+$ and $A^+ \cap B^+ = \emptyset$. 
\end{prop}

\begin{proof}
First of all, notice that $X$ contains a facing triple. This follows from the general observation:

\begin{claim}\label{claim:FacingTriple}
Let $X$ be a finite-dimensional CAT(0) cube complex. If $\overline{X}$ contains infinitely many components, then $X$ contains a facing triple.
\end{claim}

\noindent
In order to prove our claim, we anticipate on the results proved in Section~\ref{section:Horo}. Since none of the results proved in Section~\ref{section:Horo} are used in Section~\ref{section:Structure}, this does not create difficulty.

\medskip \noindent
According to Lemma~\ref{lem:FinitelyManyDirections}, up to equivalence, there exist only finitely many directions pointing to a given component, so, if $\mathfrak{R}X$ contains infinitely many components, there must exist an infinite collection $\mathscr{D}$ of pairwise non-equivalent directions. Because there cannot exist more than $\dim(X)$ pairwise transverse directions, it follows from Ramsey theorem that $\mathscr{D}$ contains an infinite subcollections $\mathscr{D}'$ of pairwise non-transverse directions. Because any $\sqsubset$-chain has length at most $\dim(X)$ according to Corollary~\ref{cor:NestedChain}, it follows from Fact~\ref{fact:DirectionTrichotomy} that $\mathscr{D}'$ contains three (in fact, infinitely many) pairwise independent directions, hence three pairwise disjoint halfspaces, i.e.\ a facing triple. This completes the proof of our claim.

\medskip \noindent
Because every hyperplane of $X$ is skewered by an element of $G$ according to \cite[Theorem~6.3]{MR2827012}, Claim~\ref{claim:FacingTriple} implies in particular that $X$ is not a quasi-line. Therefore, if we assume that $G$ does not stabilise a component in $\mathfrak{R}X$, then \cite[Lemmas~2.2 and~2.3]{PingPong} apply (thanks to Claim~\ref{claim:VisualBoundary}), proving that $X$ contains a facing triple of pairwise strongly separated hyperplanes. 
As already said, all the hyperplanes in $X$ are skewered by elements of $G$, providing the desired conclusion.

\medskip \noindent
Now, assume that $X$ is locally finite and that $G$ acts cocompactly. As a consequence of \cite[Corollary~4.9]{MR2827012}, the only case that remains to cover is when all the hyperplanes of $X$ are bounded (in fact, uniformly bounded since there exist only finitely many orbits of hyperplanes). Fix three hyperplanes in the three pairwise disjoint halfspaces delimited by a facing triple. Because $G$ acts essentially on $X$, we can choose these hyperplanes arbitrarily far from each other. If the distance between any two of these hyperplanes is larger than the maximal diameter of a hyperplane in $X$, then we get a facing triple of pairwise strongly separated hyperplanes. But, once again, we know from \cite[Theorem~6.3]{MR2827012} that all the hyperplanes in $X$ are skewered by elements in $G$, providing the desired conclusion.
\end{proof}

\subsection{Proof of the Structure Theorem}\label{section:Structure}

\begin{proof}[Proof of Theorem~\ref{thm:Structure}.]
Up to extracting a convex subcomplex, we assume without loss of generality that $G$ acts essentially on $X$ (see Proposition~\ref{prop:Essential}). Decompose $X$ as a product $L_1 \times \cdots \times L_m \times X_1 \times \cdots \times X_n$ of irreducible subcomplexes where $L_1, \ldots, L_m$ correspond to the factors having only finitely many components in their Roller completions. 

\begin{claim}\label{claim:QuasiLine}
For every $1 \leq i \leq m$, $L_i$ is a quasi-line.
\end{claim}

\noindent
Fix an index $1 \leq i \leq m$. It follows from \cite[Theorem~6.3]{MR2827012} that $L_i$ admits at least one contracting isometry $g \in \mathrm{Isom}(X)$. However, $L_i$ cannot contain two independent contracting isometries. Indeed, otherwise a standard ping-pong argument allows us to construct infinitely many pairwise independent contracting isometries. But, according to Lemma~\ref{lem:ContractingInfinity}, the points at infinity of axes of independent contracting isometries belong to distinct components, contradicting the fact that $L_i$ has only finitely many components in its Roller completion. Because we also know from \cite[Theorem~6.3]{MR2827012} that every hyperplane is skewered by a contracting isometry, we deduce that $L_i$ does not contain a facing triple. Let us deduce that, given an axis $\gamma$ of $g$, there exists a constant $K \geq 0$ such that every vertex of $L_i$ lies at distance $\leq K$ from $\gamma$. Otherwise, since we know from Lemma~\ref{lem:ContractingQuasiconvex} that the convex hull of $\gamma$ is contained in a neighbourhood of $\gamma$, there would exist an arbitrarily long descending sequence of halfspaces $D_1 \subset \cdots \subset D_k$ containing $\gamma$ (and its convex hull). Because $L_i$ has no facing triple, infinitely many hyperplanes among $\{g^k J \mid k \in \mathbb{Z}\}$, where $J$ is an arbitrary hyperplane crossing $\gamma$, have to cross the $D_i$. According to Proposition~\ref{prop:ContractingCharacterisation}, we get a contradiction with the fact that $g$ is contracting. Thus, we have proved that $L_i$ is quasi-isometric to the line $\gamma$. 

\begin{claim}\label{claim:JustOneFactor}
If $G$ is cyclically hyperbolic, then $n \leq 1$. 
\end{claim}

\noindent
Assume that $n \geq 2$. As a consequence of Proposition~\ref{prop:ForCriterion} and Lemma~\ref{lem:TreeHyp}, $X_1$ (resp. $X_2$) contains a branching and uniformly distributed tree of hyperplanes $\mathcal{H}_1$ (resp. $\mathcal{H}_2$) made of pairwise strongly separated hyperplanes. Clearly, $\mathcal{H}_1$ and $\mathcal{H}_2$ define two branching and uniformly distributed trees of hyperplanes in $X$ that are transverse. Observe that, given an $i=1,2$ and a hyperplane $J$ transverse to at least two distinct hyperplanes in $\mathcal{H}_i$, then $J$ must be a hyperplane of a factor different from $X_i$ since any two distinct hyperplanes in $\mathcal{H}_i$ are strongly separated in $X_i$, so $J$ has to be transverse to all the hyperplanes in $\mathcal{H}_i$. Thus, $\mathcal{H}_1$ and $\mathcal{H}_2$ have no transeparation. It follows from Proposition~\ref{prop:NotHyp} that $G$ is not cyclically hyperbolic, concluding the proof of Claim~\ref{claim:JustOneFactor}.

\medskip \noindent
If $n=0$, then $G$ must be virtually abelian and there is nothing to prove. From now on, we assume that $n = 1$. As a consequence of Claim~\ref{claim:QuasiLine}, we can deduce from \cite[Corollary~2.8]{MR3095714} (corrected in \cite[Proposition~2.2]{Fioravanti}) that $G$ virtually splits as a product $H \times \mathbb{Z}^m$ where $H$ acts geometrically on $X_1$. But, as a consequence of Proposition~\ref{prop:ForCriterion} (combined with \cite[Lemma~6.2]{MR2827012}), $H$ contains independent rank-one elements. It follows from \cite{MR3849623} (or alternatively \cite{BBF} or \cite[Section~6.6]{MR4057355}) that $H$ is acylindrically hyperbolic.
\end{proof}

\subsection{Horomorphisms}\label{section:Horo}

\noindent
Before proving the Tits alternative provided by Theorem~\ref{thm:TitsAlt}, we need to prove the following statement (Theorem~\ref{thm:IntroDevissage} from the introduction), which is of independent interest.

\begin{thm}\label{thm:Devissage}
 Let $G$ be a group acting on a finite-dimensional CAT(0) cube complex $X$. There exist a finite-index subgroup $H \leq G$ and a convex $H$-invariant subcomplex $Z \subset \overline{X}$ such that the following assertions hold:
\begin{itemize}
	\item if $[H,H]$ is infinite, then it acts essentially on $Z$ without virtually stabilising a component of $\mathfrak{R}Z$;
	\item every $Z$-elliptic element in $[H,H]$ is $X$-elliptic.
\end{itemize} 
\end{thm}
%\begin{thm}\label{thm:Devissage}
%Let $G$ be a group acting on a finite-dimensional CAT(0) cube complex $X$. There exist a convex subcomplex $Z \subset \overline{X}$ and subgroups $G_0=G \rhd G_1 \rhd \cdots \rhd G_k$ (with $k \leq \dim(X)$) such that the following assertions hold:
%\begin{itemize}
%	\item for every $0 \leq i \leq k-1$, $G_i/G_{i+1}$ is either finite or free abelian of finite rank;
%	\item $G_k$ acts essentially on $Z$ without virtually stabilising a component of $\mathfrak{R}Z$;
%	\item every $Z$-elliptic element in $G_k$ is $X$-elliptic.
%\end{itemize}
%\end{thm}

\noindent
The proof of the theorem is based on an elaboration of the morphisms constructed in the appendix of \cite{MR3509968}. 

\begin{definition}
Let $X$ be a CAT(0) cube complex, $x \in X$ a vertex, and $Y \subset \mathfrak{R}X$ a component. The \emph{horomorphism}\index{horomorphism} of $Y$ is the map $\mathfrak{h}_Y : \mathrm{stab}(Y) \to \mathbb{Z}$ defined by 
$$g \mapsto | \mathcal{W} (x|Y) \backslash g \mathcal{W} (x|Y) | - | g \mathcal{W}(x|Y) \backslash \mathcal{W}(x|Y) |$$
where $\mathcal{W}(x|Y)$ denotes the set of the hyperplanes of $X$ separating $x$ from $Y$. 
\end{definition}

\noindent
As proved in \cite[Proposition~4.H.1]{CornulierCommensurated}, the map thus constructed turns out to be a morphism that does not depend on the choice of the vertex $x$.

\begin{fact}\label{fact:TransferMap}
The map $\mathfrak{h}_Y$ is a well-defined morphism that does not depend on a particular choice of $x$.
\end{fact}

\noindent
However, the horomorphism does not capture sufficiently precisely the structure of the stabiliser of a component in the Roller completion. We are going to construct a \emph{full horomorphism} by combining several horomorphisms.

\begin{definition}
Let $X$ be a CAT(0) cube complex. A \emph{direction} $(D_i)_{i \geq 1}$ is a decreasing sequence $D_1 \supsetneq D_2 \supsetneq \cdots$ of halfspaces. It is a \emph{direction to $Y$}\index{direction to a component at infinity} if $Y \subset D_i$ for every $i \geq 1$.
\end{definition}

\noindent
Observe that a direction to a given component may not exist. This happens for instance in infinite cubes, since they do not even contain decreasing sequences of halfspaces. Nevertheless, as an immediate consequence of the proof of Lemma~\ref{lem:RollerDim} (or of \cite[Proposition~2.9]{MR4283595}):

\begin{fact}\label{fact:DirectionsExist}
Let $X$ be a finite-dimensional CAT(0) cube complex and $Y \subset \mathfrak{R}X$ a component. There exists a direction to $Y$.
\end{fact}

\noindent
In the sequel, given two components $Y,Z \subset \mathfrak{R}X$ in the Roller boundary of a CAT(0) cube complex $X$, we write $Y \prec Z$ if every direction to $Y$ is also a direction to $Z$.

\begin{definition}
Let $X$ be a CAT(0) cube complex and $Y \subset \mathfrak{R}X$ a component. Let $\mathcal{Y}$ denote the set of the components $Z$ satisfying $Z \prec Y$ and let $\mathrm{stab}_0(Y)$ denote the subgroup of the stabiliser of $Y$ in $\mathrm{Isom}(X)$ that stabilises all the components in $\mathcal{Y}$. The \emph{full horomorphism}\index{full horomorphism} of $Y$ is
$$\mathfrak{H}_Y : \left\{ \begin{array}{ccc} \mathrm{stab}_0(Y) & \to & \mathbb{Z}^\mathcal{Y} \\ g & \mapsto & \left( \mathfrak{h}_Z(g) \right)_{Z \in \mathcal{Y}} \end{array} \right.$$
\end{definition}

\noindent
Our main result about full horomorphism is the following:

\begin{thm}\label{thm:KernelFullHoro}
Let $X$ be a finite-dimensional CAT(0) cube complex and $Y \subset \mathfrak{R}X$ a component. Then $\mathrm{stab}_0(Y)$ has finite index in $\mathrm{stab}(Y)$, and an element in $\mathrm{ker}(\mathfrak{H}_Y)$ is $X$-elliptic if and only if it is $Y$-elliptic. 
\end{thm}

\noindent
The theorem will be a straightforward consequence of Propositions~\ref{prop:FullHoro} and~\ref{prop:FinitelyManyComponents} below. 

\begin{prop}\label{prop:FullHoro}
Let $X$ be a finite-dimensional CAT(0) cube complex, $Y \subset \mathfrak{R}X$ a component, and $g \in \mathrm{Isom}(X)$ a loxodromic isometry stabilising $Y$. Either $g$ has unbounded orbits in $Y$ or there exists a component $Z \prec Y$ stabilised by some power $g^k$ such that $\mathfrak{h}_Z(g^k) \neq 0$. 
\end{prop}

\begin{proof}
Fix an axis $\gamma$ of $g$. Because $X$ is finite-dimensional, we can find an arbitrarily large collection of pairwise non-transverse hyperplanes crossing $\gamma$. Because there exist only finitely many $\langle g \rangle$-orbits of hyperplanes crossing $\gamma$, it follows that, up to replacing $g$ with one of its powers, there exists a hyperplane $J$ such that $\mathcal{V} := \{g^k \cdot J \mid k \in \mathbb{Z}\}$ is a collection of pairwise non-transverse hyperplanes crossing~$\gamma$. 

\medskip \noindent
Observe that, if there exists a hyperplane in $\mathcal{V}$ crossing $Y$, then $g$ has unbounded orbits in $Y$. Indeed, because $Y$ is $\langle g \rangle$-invariant, all the hyperplanes in $\mathcal{V}$ would have to cross $Y$, and we would find a halfspace in $Y$ such that $gD \supsetneq D$. Consequently, we assume from now on that no hyperplane in $\mathcal{V}$ crosses $Y$. For convenience, we fix a left-right order on $\gamma$ such that, if we denote by $V^+$ the halfspace delimited by any $V \in \mathcal{V}$ that contains the right part of $\gamma$, then $Y \subset V^+$. Up to replacing $g$ with its inverse, we assume that $g$ acts on $\gamma$ as a positive translation.

\medskip \noindent
Now, define $Z \subset \mathfrak{R}X$ as the component containing the orientations
$$\sigma_{\mathcal{V},x} : H \mapsto \left\{ \begin{array}{l} \text{halfspace containing a $V^+$ if it exists} \\ \text{halfspace containing $x$ otherwise} \end{array} \right., \ x \in X.$$
It is straightforward to check that each $\sigma_{\mathcal{V},x}$ is an orientation and that:

\begin{fact}\label{fact:SigmaVx}
For all $x,y \in X$, the hyperplanes on which $\sigma_{\mathcal{V},x}$ and $\sigma_{\mathcal{V},y}$ disagree are exactly the hyperplanes separating $x$ and $y$ that cross all the halfspaces in $\{V^+ \mid V \in \mathcal{V} \}$.
\end{fact}

\noindent
So the $\sigma_{\mathcal{V},x}$ all belong to a common a component of $\mathfrak{R}X$, and we choose this component for $Z$. Observe that $g$ stabilises $\{\sigma_{\mathcal{V},x} \mid x \in X\}$ because $\mathcal{V}$ is $\langle g \rangle$-invariant, so $g$ stabilises $Z$. We begin by proving that every direction to $Z$ is a direction to $Y$. So let $D:=(D_i)_{i \geq 1}$ be a direction to $Z$. If it is not a direction to $Y$, then there must exist some $i \geq 1$ such that $D_i$ dos not contain $Y$. But all the $V^+$ contain both $Z$ and $Y$, so the hyperplane $H_i$ delimiting $D_i$ must cross all the $V^+$. We deduce from Fact~\ref{fact:SigmaVx} that, if $x,y \in X$ are any two vertices separated by $H_i$, then $\sigma_{\mathcal{V},x}$ and $\sigma_{\mathcal{V},y}$ are separated by $H_i$, which is impossible since $H_i$ cannot cross $Z$. Thus, $D$ has to define a direction to $Y$.

\medskip \noindent
Next, we want to prove that $\mathfrak{h}_Z(g) \neq 0$. We begin by observing that:

\begin{claim}\label{claim:ForHoro}
Fix two vertices $x,y \in \gamma$ with $y$ on the right of $x$. No hyperplane separating $x$ from $y$ can separate $y$ from $Z$. 
\end{claim}

\noindent
It suffices to show that a hyperplane $H$ separating $y$ from $Z$ has to cross $\gamma$ on the right of $y$. Then, it will not be able to separate $x$ and $y$ since otherwise it would cross $\gamma$ twice. Assume for contradiction that $H$ does not intersect the subsegment $\gamma^+ \subset \gamma$ that lies on on the right of $y$. So $H$ separates $\gamma^+$ and $Z$. Because the $V^+$ all intersect both $\gamma^+$ and $Z$, it follows that $H$ crosses all the $V^+$. We deduce from Fact~\ref{fact:SigmaVx} that, if $a,b \in X$ are any two vertices separated by $H$, then $\sigma_{\mathcal{V},a}$ and $\sigma_{\mathcal{V},b}$ are separated by $H$, which is impossible since $H$ does not cross $Z$. This proves Claim~\ref{claim:ForHoro}

\medskip \noindent
As a consequence of Claim~\ref{claim:ForHoro}, we have 
$$\mathfrak{h}_Z(g^2) = |\mathcal{W}(x|g^2x) \cap \mathcal{W}(x|Z)| \geq |\mathcal{W}(x|g^2x) \cap \mathcal{V}|\geq 1,$$
where we chose an arbitrary vertex $x \in \gamma$ (see Fact~\ref{fact:TransferMap}). We conclude that $\mathfrak{h}_Z(g) \neq 0$, as desired. 
\end{proof}

\begin{prop}\label{prop:FinitelyManyComponents}
Let $X$ be a finite-dimensional CAT(0) cube complex and $Y \subset \mathfrak{R}X$ a component. There exist only finitely many components $Z$ satisfying $Z \prec Y$.
\end{prop}

\noindent
Before turning to the proof of the proposition, we need some preliminary work.

\begin{definition}
Let $X$ be a CAT(0) cube complex. Two directions $A=(A_i)_{i \geq 1}$ and $B=(B_i)_{i \geq 1}$ are:
\begin{itemize}
	\item \emph{independent} if $A_i \cap B_j = \emptyset$ for some $i,j \geq 1$;
	\item \emph{transverse} if there exists some $N \geq 1$ such that $A_i$ and $B_j$ are transverse for all $i,j \geq N$;
	\item \emph{nested}, denoted by $A \sqsubset B$, if, for every $i \geq 1$, there exists some $j \geq 1$ such that $A_j \subset B_i$;
	\item \emph{equivalent} if $A \sqsubset B$ and $B \sqsubset A$.
\end{itemize}
\end{definition}

\noindent
Let us start with an easy observation:

\begin{fact}\label{fact:DirectionTrichotomy}
Let $X$ be a CAT(0) cube complex and $A,B$ two directions. If $A$ and $B$ are neither independent nor transverse, then $A \sqsubset B$ or $B \sqsubset A$. 
\end{fact}

\begin{proof}
Write $A$ as $(A_i)_{i \geq 1}$ and $B$ as $(B_i)_{i \geq 1}$. Because $A$ and $B$ are neither independent nor transverse, we know that, for every $n \geq 1$, there exist $i(n),j(n) \geq n$ such that $A_{i(n)} \subset B_{j(n)}$ or $A_{i(n)} \subset B_{j(n)}$. If the first inclusion occurs for infinitely many indices, then, given a $k \geq 1$, there exists some $n \geq 1$ such that $i(n),j(n) \geq k$ and $A_{i(n)} \subset B_{j(n)}$, hence $B_k \supset B_{j(n)} \supset A_{i(n)}$. Thus, we have $A \sqsubset B$. Similarly, if the second inclusion occurs for infinitely many indices, then $B \sqsubset A$. 
\end{proof}

\noindent
The first key property we have to show is that the length of a chain with respect to $\sqsubset$ is always bounded above the dimension of the cube complex under consideration. This will be a consequence of the following observation:

\begin{lemma}\label{lem:NestedTransverse}
Let $X$ be a CAT(0) cube complex and $A=(A_i)_{i \geq 1},B=(B_i)_{i \geq 1}$ two directions. If $A,B$ are not equivalent but $A \sqsubset B$, then there exists some $p \geq 1$ such that, for every $i \geq p$, there exists some $q \geq 1$ such that $A_i$ is transverse to $B_j$ for every $j \geq q$. 
\end{lemma}

\noindent
The conclusion should be read as follows: up to removing finitely many halfspaces from $A$, every halfspace in $A$ is transverse to all but finitely many halfspaces in $B$.

\begin{proof}[Proof of Lemma \ref{lem:NestedTransverse}.]
Because $A\sqsubset B$, we know that, for every $n \geq 1$, there exists some $j(n) \geq 1$ such that $B_n \supset A_{j(n)}$. And, because we do not have $B \sqsubset A$, we know that there exists some $p \geq 1$ such that $A_p$ does not contain any $B_j$. Fix some $i \geq p$. Up to removing finitely many halfspaces from $B$, we assume without loss of generality that $j(n) \geq i$ for every $n \geq 1$. Let $n$ the smallest index such that $B_n$ is not transverse to $A_i$. Then, for every $m \geq 1$ such that $B_m$ is not transverse to $A_i$, we must have $A_i \subset B_m \subset B_n$. Because there can exist only finitely many such intermediate halfspaces, it follows that $B_m$ is transverse to $A_i$ for all but finitely many $m$.
\end{proof}

\begin{cor}\label{cor:NestedChain}
Let $X$ be a CAT(0) cube complex and $A_1, \ldots, A_k$ a collection of pairwise non-equivalent directions. If $A_1 \sqsubset \cdots \sqsubset A_k$, then $k \leq \dim(X)$.
\end{cor}

\begin{proof}
As a consequence of Lemma~\ref{lem:NestedTransverse}, we know that, for every $1 \leq i \leq k-1$, all but finitely many halfspaces in $A_i$ are transverse to all but finitely many hyperplanes in $A_{i+1} \cup \cdots \cup A_k$. Therefore, there exist halfspaces $D_1 \in A_1, \ldots, D_k \in A_k$ that are pairwise transverse, hence $k \leq \dim(X)$ as desired.
\end{proof}

\noindent
The second key idea we want to show is that a component in the Roller boundary of a CAT(0) cube complex is uniquely determined by the directions pointing to it. This will be a consequence of the following observation:

\begin{lemma}\label{lem:DirectionPrec}
Let $X$ be a finite-dimensional CAT(0) cube complex and $Y,Z \subset \mathfrak{R}X$ two components. Then $Y$ lies in $\overline{Z}$ if and only if every direction to $Z$ is a direction to $Y$.
\end{lemma}

\begin{proof}
Assume that every direction to $Z$ is a direction to $Y$ and let $\sigma \in Y$ be an orientation of $X$. Saying that $\sigma$ belongs to $\overline{Z}$ amounts to saying that, for every hyperplane $J$ of $X$ that does not cross $Z$, $\sigma(J)$ is the halfspace $D$ delimited by $J$ containing $Z$. 

\medskip \noindent
Fix a vertex $x \in N(J)$ and let $\xi \in Z$ the projection of $x$ on $Z$ as given by Lemma~\ref{lem:ProjBoundary}. If all the hyperplanes separating $x$ from $\xi$ (and, a fortiori, from $Z$) are transverse to $J$, then 
$$\xi' : H \mapsto \left\{ \begin{array}{cl} \xi(H) & \text{if $H \neq J$} \\ \xi(H)^c & \text{otherwise} \end{array} \right.$$
defines an orientation. By construction, $\xi'$ differ with $\xi$ exactly on $J$, so $\xi'$ would provide a neighbour of $\xi$ in $Z$ such that $\xi,\xi'$ are separated by $J$, contradicting the fact that $J$ does not cross $Z$. Thus, there exists a hyperplane $J_1$ separating $J$ and $Z$. Let $D_1$ denote the halfspace delimited by $J_1$ containing $Z$. 

\medskip \noindent
By iterating the construction, one gets a direction $D \supsetneq D_1 \supsetneq D_2 \supsetneq \cdots$ to $Z$. But it has to be a direction to $Y$ as well, hence $Y \subset D$, which implies that $\sigma(J)=D$ as desired.

\medskip \noindent
Conversely, assume that $Y \subset \overline{Z}$ and let $D=(D_i)_{i \geq 1}$ be a direction to $Z$. Fix an $i \geq 1$. Saying that $Z$ lies in $D_i$ amounts to saying that $\sigma(J_i)=D_i$ for every orientation $\sigma\in Z$ of $X$, where $J_i$ denotes the hyperplane delimiting $D_i$. By definition of $\overline{Z}$ (as a subspace of $\overline{X}$), we must also have $\sigma(J_i)=D_i$ for every orientation $\sigma \in Y$, hence $Y \subset D_i$. Therefore, $D$ is a direction to $Y$ as well.
\end{proof}

\begin{cor}\label{cor:DeterminedDirection}
Let $X$ be a finite-dimensional CAT(0) cube complex. For any two distinct components of $\mathfrak{R}X$, there exists a direction to one of them that is not a direction to the~other.
\end{cor}

\begin{proof}
Assume that $Y,Z \subset \mathfrak{R}X$ are two components having the same directions. By applying Lemma~\ref{lem:DirectionPrec} twice, it follows that $Y \subset \overline{Z}$ and $Z \subset \overline{Y}$. We conclude from Lemma~\ref{lem:RollerDim} that $Y=Z$. 
\end{proof}

\noindent
By putting everything together, we can prove the following statement, which will essentially prove Proposition~\ref{prop:FinitelyManyComponents}.

\begin{lemma}\label{lem:FinitelyManyDirections}
Let $X$ be a finite-dimensional CAT(0) cube complex and $Y \subset \mathfrak{R} X$ a component. Up to equivalence, there exist only finitely many directions to $Y$.
\end{lemma}

\begin{proof}
Assume for contradiction that there exist infinitely many pairwise non-equivalent directions $D_1,D_2, \ldots$ to $Y$. Because at most $\dim(X)$ directions can be pairwise transverse, it follows from Ramsey theorem that we can assume that $D_1,D_2, \ldots$ are also pairwise non-transverse. We deduce from Fact~\ref{fact:DirectionTrichotomy} that $D_1,D_2, \ldots$ are pairwise $\sqsubset$-related. Consequently, we find arbitrarily long $\sqsubset$-chains, contradicting Corollary~\ref{cor:NestedChain}.
\end{proof}

\begin{proof}[Proof of Proposition~\ref{prop:FinitelyManyComponents}.]
According to Lemma~\ref{lem:FinitelyManyDirections}, there exist only finitely many directions to $Y$ up to equivalence. On the other hand, we know from Corollary~\ref{cor:DeterminedDirection} that a component of $\mathfrak{R}X$ is uniquely determined by the directions pointing to it, up to equivalence; and, according to Lemma~\ref{lem:DirectionPrec}, every component containing $Y$ in its Roller completion has its directions pointing to $Y$. We conclude that there may exist only finitely many components containing $Y$ in their completions. 
\end{proof}

\begin{proof}[Proof of Theorem~\ref{thm:KernelFullHoro}.]
The fact that $\mathrm{stab}_0(Y)$ has finite index in $\mathrm{stab}(Y)$ is an immediate consequence of Proposition~\ref{prop:FinitelyManyComponents}. Next, fix an element $g \in \mathrm{stab}_0(Y)$. If $g$ is not $X$-elliptic, then there exists some $k \geq 1$ such that $g^k$ is $X$-loxodromic. We deduce from Proposition~\ref{prop:FullHoro} that $g^k$, and a fortiori $g$, has unbounded orbits in $Y$. Conversely, assume that $g$ is not $Y$-elliptic. Because $Y$ is finite-dimensional according to Lemma~\ref{lem:RollerDim}, one can find a halfspace $D$ and a power $k \geq 1$ such that $g^{k}D \subsetneq D$. But halfspaces of $Y$ can be thought of as halfspaces of $X$, so the same property must be satisfied in $X$, proving that $g$ is not $X$-elliptic either.
\end{proof}

\noindent
We are finally ready to prove the main result of this section, namely Theorem~\ref{thm:Devissage}.  We start by recording the following consequence of our understanding of directions:

\begin{lemma}\label{lem:NoCompStab}
Let $X$ be a finite-dimensional CAT(0) cube complex. If $a$ and $b$ are two independent contracting isometries, then no component of $\mathfrak{R}X$ is virtually stabilised by both $a$ and $b$.
\end{lemma}

\begin{proof}
As a consequence of Lemma~\ref{lem:ContractingInfinity}, it suffices to prove that, given a contracting isometry $g \in \mathrm{Isom}(X)$, the only components of $\mathfrak{R}X$ that are virtually stablised by $g$ are the two components that contain the points at infinity of a given axis $\gamma$ of $g$. Recall from Proposition~\ref{prop:ContractingCharacterisation} that there exist a constant $L \geq 0$ and a hyperplane $J$ such that $\mathcal{J}:= \{g^nJ \mid n \in \mathbb{Z}\}$ is a collection of pairwise $L$-well-separated hyperplanes. 

\medskip \noindent
Let $Y \subset \mathfrak{R}X$ be a component. First, of all, notice that $Y$ is crossed by at most one hyperplane from $\mathcal{J}$. Indeed, if $J_1,J_2 \in \mathcal{J}$ both cross $Y$ and if $H_1,H_2, \ldots$ are the hyperplanes given by a direction to $Y$ (which exists according to Fact~\ref{fact:DirectionsExist}), then $J_1$ and $J_2$ have to be transverse to all but finitely many of the $H_i$. Since $J_1$ and $J_2$ are $L$-well-separated, this is not possible. Thus, there are two cases to consider.

\medskip \noindent
First, assume that $Y$ lies between $g^iJ$ and $g^{i+2}J$ for some $i \in \mathbb{Z}$. Then each $g^{2k}Y$ lies between $g^{i+2k}J$ and $g^{i+2(k+1)}J$, which implies that the $g^{2k}Y$ are pairwise distinct. A fortiori, $g$ does not virtually stabilise $Y$.

\medskip \noindent
Next, assume that $Y$ is contained in $g^nJ^+$ for every $n \in \mathbb{Z}$, where $J^+$ is a halfspace delimited by $J$. Up to replacing $g$ with $g^{-1}$, we assume that $gJ^+ \subset J^+$. Notice that, given an arbitrary component $W \subset \mathfrak{R}X$ satisfying $W \subset g^nJ^+$ for every $n \in \mathbb{Z}$ and a direction $(D_i)_{i \geq 0}$ to $W$, necessarily $(D_i)_{i \geq 0}$ is equivalent to $(g^iJ^+)_{i \geq 0}$. Otherwise, it would follow from Fact~\ref{fact:DirectionTrichotomy} and Lemma~\ref{lem:NestedTransverse} that there exist $j \geq 0$ such that $g^jJ$ and $g^{j+1}J$ are transverse to more than $L$ hyperplanes delimited the $D_i$, which is impossible since $g^jJ$ and $g^{j+1}J$ are $L$-well-separated. We conclude that there exists a unique component of $\mathfrak{R}X$ contained in all the $g^nJ^+$, namely $Y$, which has to coincide with a component of $\mathfrak{R}X$ containing a point at infinity of an axis of $g$.
\end{proof} 

\begin{proof}[Proof of Theorem~\ref{thm:Devissage}.]
 Let $Y$ be a component of $\overline{X}$ whose orbit under $G$ is finite. We choose $Y$ with the smallest possible dimension. Let $H \leq G$ be a finite-index subgroup contained in $\mathrm{stab}_0(Y)$. (Recall that, as a consequence of Theorem~\ref{thm:KernelFullHoro}, $\mathrm{stab}_0(Y)$ has finite index in $\mathrm{stab}(Y)$.) By construction, $H$ cannot stabilise a component of $\mathfrak{R}Y$, so it follows from Proposition~\ref{prop:Essential} that there is a convex subcomplex $E \subset Y$ on which $H$ acts essentially. Decompose $E$ as a product $E_1 \times \cdots \times E_n$ of irreducible subcomplexes. Up to replacing $H$ with a finite-index subgroup, we can assume that $H$ preserves the product structure of $E$.

\medskip \noindent
Fix an index $1 \leq i \leq n$. It follows from our choice of $Y$ that $H$ cannot virtually stabilise a component of $\mathfrak{R}E_i$. As a consequence of Proposition~\ref{prop:ForCriterion}, combined with Proposition~\ref{prop:ContractingCharacterisation}, $H$ contains two independent contracting isometries $a_i$ and $b_i$, with respect to the action of $H$ on $E_i$. We know from Lemma~\ref{lem:PurelyContractingFree} that, up to replacing $a_i,b_i$ with some of their powers, $F_i:= \langle a_i ,b_i \rangle$ is a purely contracting free subgroup of rank two and $[F_i,F_i]$ must contain two independent contracting elements. In particular, this implies, according to Lemma~\ref{lem:NoCompStab}, that $[F_i,F_i]$ does not virtually stabilise a component of $\mathfrak{R}E_i$. A fortiori, $[H,H]$ does not virtually stabilise a component of $\mathfrak{R}E_i$, and we deduce from Proposition~\ref{prop:Essential} that there exists a convex subcomplex $Z_i \subset E_i$ on which $[H,H]$ acts essentially. 

\medskip \noindent
So far, we know that $[H,H]$ acts essentially on the convex subcomplex $Z:= Z_1 \times \cdots \times Z_n \subset E$. Now, observe that $[H,H]$ is contained in the kernel of the full horomorphism of $Y$. Therefore, Theorem~\ref{thm:KernelFullHoro} implies that an element of $[H,H]$ is $X$-elliptic if and only if it is $Z$-elliptic. 
%We construct a sequence of subgroups $G_0=G \rhd G_1 \rhd \cdots$ and of convex subcomplexes $X_0=X, X_1, \ldots \subset \overline{X}$ that are respectively $G_0-,G_1-,\ldots-$invariant as follows. Assume that $G_i$ and $X_i$ are constructed for some $i \geq 0$. If $G_i$ contains a finite-index subgroup that stabilises a component $Y \subset \mathfrak{R}X_i$, we set:
%\begin{itemize}
%	\item $G_{i+1} \leq G_i \cap \mathrm{stab}_0(Y)$ is a finite-index normal subgroup in $G_i$ stabilising $Y$ and $X_{i+1}=X_i$;
%	\item $G_{i+2} = \mathrm{ker}(\mathfrak{H}_Y) \cap G_{i+1}$ and $X_{i+2}=Y$.
%\end{itemize}
%As a consequence of Lemma~\ref{lem:RollerDim}, the process has to stop after $k \leq \dim(X)$ steps. Moreover, we deduce from Theorem~\ref{thm:KernelFullHoro} that, for every $0 \leq i \leq k-1$: 
%\begin{itemize}
%	\item an element in $G_{i+1}$ that is $X_{i+1}$-elliptic must be also $X_i$-elliptic;
%	\item $G_i/G_{i+1}$ is either finite or free abelian of finite rank.
%\end{itemize}
%Therefore, $G_k$ acts on $X_k$ without virtually stabilising a component of $\mathfrak{R}X_k$ and every $X_k$-elliptic element in $G_k$ is also $X$-elliptic. It follows from Proposition~\ref{prop:Essential} that $G_k$ acts essentially on some convex subcomplex $Z \subset X_k$, which is the subcomplex we are looking~for.
\end{proof}

\subsection{Proof of the Strong Tits Alternative}

\begin{proof}[Proof of Theorem~\ref{thm:TitsAlt}.]
 Let $X$ be a CAT(0) cube complex on which $G$ acts geometrically and let $K \leq G$ be an arbitrary subgroup. According to Theorem~\ref{thm:Devissage}, there exists a finite-index subgroup $H \leq K$ and a convex $H$-invariant subcomplex $Z \subset \overline{X}$ such that: 
\begin{itemize}
	\item if $[H,H]$ is infinite, then it acts essentially on $Z$ without virtually stabilising a component of $\mathfrak{R}Z$;
	\item every $Z$-elliptic element in $[H,H]$ is $X$-elliptic.
\end{itemize}
The second item implies that $[H,H]$ acts on $Z$ with torsion stabilisers. But we know that every torsion subgroup of $G$ is finite. Indeed, it follows from the proof of \cite[Theorem~5.1]{MR1347406} that every infinite finitely generated subgroup of $G$ contains an infinite-order element, so every torsion subgroup of $G$ must be an increasing union of finite subgroups. Since $G$ contains only finitely many conjugacy classes of finite subgroups (as any CAT(0) groups), the desired conclusion follows. As a consequence, $[H,H]$ acts on $Z$ properly. Therefore, if $Z$ is reduced to a single vertex, then $[H,H]$ must be finite and $H$ is virtually abelian. From now on, we assume that $Z$ is not reduced to a single vertex. Decompose $Z$ as a product of irreducible unbounded subcomplexes $Z_1 \times \cdots \times Z_n$ and fix a finite-index subgroup $Q \leq [H,H]$ that preserves the product structure. This is possible according to \cite[Proposition~2.6]{MR2827012}.

\medskip \noindent
Given an index $1 \leq i \leq n$, $Q$ acts essentially on $Z_i$ without virtually stabilising a component in $\mathfrak{R}Z_i$. (In particular, this implies that $\overline{Z_i}$ has infinitely many components.) Thus, the assumptions of Proposition~\ref{prop:ForCriterion} hold, and we can apply Lemma~\ref{lem:TreeHyp}, which proves that $Z_i$ contains a branching and uniformly distributed tree of pairwise strongly separated hyperplanes $\mathcal{H}_i$. A fortiori, $\mathcal{H}_i$ is a branching and uniformly distributed tree of hyperplanes in $X$. Moreover, if $J$ is a hyperplane of $X$ that separates two hyperplanes in $\mathcal{H}_i$, then it has to be a hyperplane of $Z$ and even a hyperplane of $Z_i$ (by convexity of $Z$), which implies that $J$ cannot be transverse to two hyperplanes in $\mathcal{H}_i$ since the hyperplanes in $\mathcal{H}_i$ are pairwise strongly separated. Thus, $\mathcal{H}_i$ has no transeparation. 

\medskip \noindent
Because the trees of hyperplanes $\mathcal{H}_i$ and $\mathcal{H}_j$ are clearly transverse for all distinct $1 \leq i,j \leq n$, it follows from Proposition~\ref{prop:NotHyp} that necessarily $n=1$. In other words, $Z$ is irreducible. It follows that the elements $a,b \in [H,H]$ given by Proposition~\ref{prop:ForCriterion} are independent rank-one isometries, which implies that $[H,H]$ is acylindrically hyperbolic (by applying \cite{MR3849623}, or alternatively \cite[Example~2.1(3)]{BBF} or \cite[Section~6.6]{MR4057355}).  
\end{proof}

\section{Cyclically hyperbolic special groups}

\subsection{Combinatorics of special cube complexes}

\noindent
In this section, we recall the definition of special cube complexes and record some results required for the proof of Theorem~\ref{thm:Special} in the next section. 

\begin{definition}
Let $J$ be a hyperplane with a fixed orientation $\vec{J}$ in a (nonpositively curved) cube complex. We say that $J$ is:
\begin{itemize}
	\item \textit{2-sided} if $\vec{J} \neq \vec{J}^{-1}$, where $\vec{J}^{-1}$ denote opposite orientation of $\vec{J}$;
	\item \textit{self-intersecting} if there exist two edges dual to $J$ which are non-parallel sides of some square;
	\item \textit{self-osculating} if there exist two oriented edges dual to $\vec{J}$ with the same initial points or the same terminal points, but which do not belong to a same square.
\end{itemize}
Moreover, if $H$ is another hyperplane, then $J$ and $H$ are:
\begin{itemize}
	\item \textit{transverse} if there exist two edges dual to $J$ and $H$ respectively which are non-parallel sides of some square,
	\item \textit{inter-osculating} if they are transverse, and if there exist two edges dual to $J$ and $H$ respectively with one common endpoint, but which do not belong to a same square.
\end{itemize}
\end{definition}
\noindent
Sometimes, one refers 1-sided, self-intersecting, self-osculating and inter-osculating hyperplanes as \textit{pathological configurations of hyperplanes}. The last three configurations are illustrated on Figure \ref{figure17}.
\begin{figure}
\begin{center}
\includegraphics[trim={0, 19.5cm 1cm 0},clip,scale=0.45]{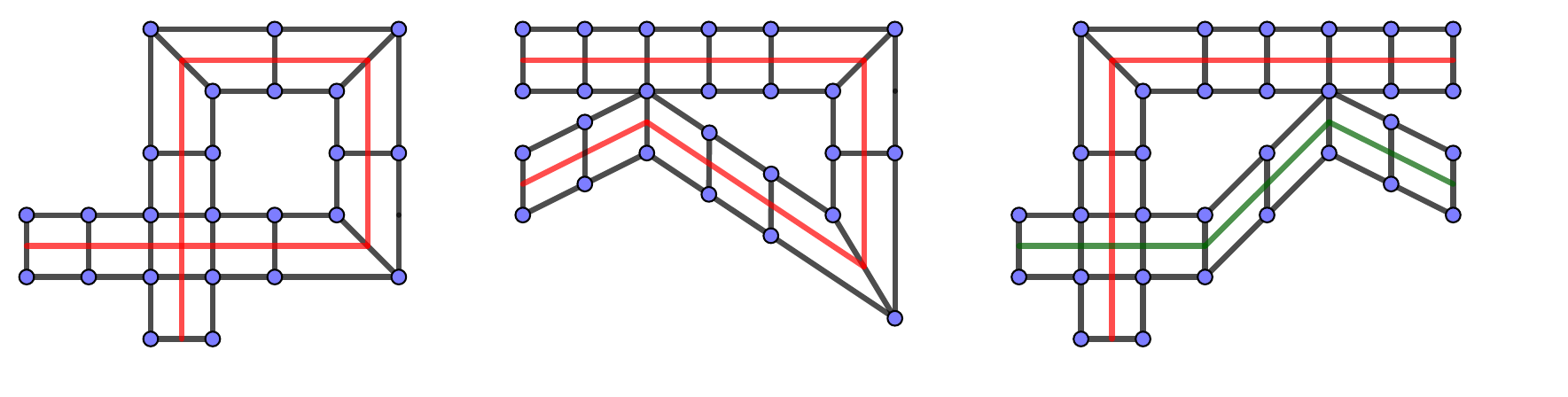}
\caption{From left to right: self-intersection, self-osculation, inter-osculation.}
\label{figure17}
\end{center}
\end{figure}

\begin{definition}
A \emph{special cube complex} is a nonpositively curved cube complex whose hyperplanes are two-sided and which does not contain self-intersecting, self-osculating or inter-osculating hyperplanes. A group which can be described as the fundamental group of a (compact) special cube complex is \emph{(cocompact) special}.
\end{definition}

\noindent
Let $X$ be a cube complex (not necessarily nonpositively-curved). For us, a \emph{path} in $X$ is the data of a sequence of successively adjacent edges. What we want to understand is when two such paths are homotopic (with fixed endpoints). For this purpose, we need to introduce the following elementary transformations. One says that:  
\begin{itemize}
	\item a path $\gamma \subset X$ contains a \emph{backtrack} if the word of oriented edges read along $\gamma$ contains a subword $ee^{-1}$ for some oriented edge $e$;
	\item a path $\gamma' \subset X$ is obtained from another path $\gamma \subset X$ by \emph{flipping a square} if the word of oriented edges read along $\gamma'$ can be obtained from the corresponding word of $\gamma$ by replacing a subword $e_1e_2$ with $e_2'e_1'$ where $e_1',e_2'$ are opposite oriented edges of $e_1,e_2$ respectively in some square of $X$. 
\end{itemize}
We claim that these elementary transformations are sufficient to determine whether or not two paths are homotopic. More precisely:

\begin{prop}\label{prop:cubehomotopy}
Let $X$ be a cube complex and $\gamma,\gamma' \subset X$ two paths with the same endpoints. The paths $\gamma,\gamma'$ are homotopic (with fixed endpoints) if and only if $\gamma'$ can be obtained from $\gamma$ by removing or adding backtracks and flipping squares. 
\end{prop}

\noindent
This statement follows from the fact that flipping squares provide the relations of the fundamental groupoid of $X$; see \cite[Statement 9.1.6]{BrownGroupoidTopology} for more details. 

\medskip \noindent
In the rest of this subsection, we describe (a particular case of) the formalism introduced in \cite{SpecialRH}. 

\medskip \noindent
From now on, we fix a (not necessarily compact) special cube complex $X$. The \emph{crossing graph}\footnote{The graph $\Delta X$ will define the right-angled Artin group in which the fundamental group of $X$ will embed. This is the same graph used in \cite{MR2377497}, where no name is assigned to it. Here, we use the suggestive terminology \emph{crossing graph}, but it should not be confused with related but different graphs introduced in various contexts, e.g.\ \cite{MR1153934, MR1379364, Roller, MR1920184, MR3217625}.} $\Delta X$ is the graph whose vertices are the hyperplanes of $X$ and whose edges link two transverse hyperplane. %Similarly, the \emph{oriented crossing graph} $\vec{\Delta} X$ is the graph whose vertices are the oriented hyperplanes of $X$ and whose edges link two oriented hyperplanes whenever their underlying unordered hyperplanes are transverse.

\medskip \noindent
Given a vertex $x_0 \in X$, a word $w=J_1 \cdots J_r$ of oriented hyperplanes $J_1, \ldots, J_r$ is \emph{$x_0$-legal} if there exists a path $\gamma$ in $X$ starting from $x_0$ such that the oriented hyperplanes it crosses are successively $J_1, \ldots, J_r$. We say that the path $\gamma$ \emph{represents} the word $w$. 

\begin{fact}\label{fact:uniquepath}\emph{\cite[Fact 3.3]{SpecialRH}}
An $x_0$-legal word is represented by a unique path in $X$. 
\end{fact}

\noindent
The previous fact allows us to define the \emph{terminus} of an $x_0$-legal word $w$, denoted by $t(w)$, as the ending point of the unique path representing $w$. 

\medskip \noindent
Set $\mathcal{L}(X) = \{ x \text{-legal words} \mid x \in X\}$ the set of all legal words. (If $x_1,x_2 \in X$ are two distinct points, we consider the empty $x_1$-legal and $x_2$-legal words as distinct.) We consider the equivalence relation $\sim$ on $\mathcal{L}(X)$ generated by the following transformations:
\begin{description}
	\item[(cancellation)] if a legal word contains $JJ^{-1}$ or $J^{-1}J$, remove this subword;
	\item[(insertion)] given an oriented hyperplane $J$, insert $JJ^{-1}$ or $J^{-1}J$ as a subword of a legal word;
	\item[(commutation)] if a legal word contains $J_1J_2$ where $J_1,J_2$ are two transverse oriented hyperplanes, replace this subword with $J_2 J_1$. 
\end{description}
So two $x$-legal words $w_1,w_2$ are equivalent with respect to $\sim$ if there exists a sequence of $x$-legal words 
$$m_1=w_1, \ m_2, \ldots, m_{r-1}, \ m_r=w_2$$ 
such that $m_{i+1}$ is obtained from $m_i$ by a cancellation, an insertion, or a commutation for every $1 \leq i \leq r-1$. Define $\mathcal{D}(X)= \mathcal{L}(X)/ \sim$ as a set of \emph{diagrams}. The following observation allows us (in particular) to define the \emph{terminus} of a diagram as the terminus of one of the legal words representing it.

\begin{fact}\label{fact:transformations}\emph{\cite[Fact 3.4]{SpecialRH}}
Let $w'$ be an $x_0$-legal word obtained from another $x_0$-legal word $w$ by a cancellation / an insertion / a commutation. If $\gamma',\gamma$ are paths representing $w',w$ respectively, then $\gamma'$ is obtained from $\gamma$ by removing a backtrack / adding a backtrack / flipping a square. 
\end{fact}

\noindent
In the sequel, an \emph{$(x,y)$-diagram} will refer to a diagram represented by an $x$-legal word with terminus $y$, or just an \emph{$(x,\ast)$-diagram} if we do not want to specify its terminus. A diagram which is an $(x,x)$-diagram for some $x \in X$ is \emph{spherical}. 

\medskip \noindent
If $w$ is an $x_0$-legal word and $w'$ a $t(w)$-legal word, we define the \emph{concatenation} $w \cdot w'$ as the word $ww'$, which is $x_0$-legal since it is represented by the concatenation $\gamma \gamma'$ where $\gamma, \gamma'$ represent respectively $w,w'$. Because we have the natural identifications
\begin{table}[h]
	\centering
		\begin{tabular}{cccc}
			$\mathcal{L}(X)$ & $\leftrightarrow$ & paths in $X$ & (Fact \ref{fact:uniquepath}) \\ 
			$\sim$ & $\leftrightarrow$ & homotopy with fixed endpoints & (Fact \ref{fact:transformations}, Proposition \ref{prop:cubehomotopy}) 
		\end{tabular}
\end{table}

\noindent
it follows that the concatenation in $\mathcal{L}(X)$ induces a well-defined operation in $\mathcal{D}(X)$, making $\mathcal{D}(X)$ isomorphic to the fundamental groupoid of $X$. As a consequence, if we denote by $M(X)$ the Cayley graph of the groupoid $\mathcal{D}(X)$ with respect to the generating set given by the oriented hyperplanes, and, for every $x \in X$, $M(X,x)$ the connected component of $M(X)$ containing the trivial path $\epsilon(x)$ based at $x$, and $\mathcal{D}(X,x)$ the vertex-group of $\mathcal{D}(X)$ based at $\epsilon(x)$, then the previous identifications induce the identifications
\begin{table}[h]
	\centering
		\begin{tabular}{ccc}
			$\mathcal{D}(X)$ & $\leftrightarrow$ & fundamental groupoid of $X$  \\ 
			$\mathcal{D}(X,x)$ & $\leftrightarrow$ & $\pi_1(X,x)$  \\ 
			$M(X,x)$ & $\leftrightarrow$ & universal cover $\widetilde{X}^{(1)}$ with a specified basepoint  \\ 
		\end{tabular}
\end{table}

\noindent
More explicitly, $\mathcal{D}(X,x)$ is the group of $(x,x)$-diagrams endowed with the concatenation, and $M(X,x)$ is the graph whose vertices are the $(x,\ast)$-diagrams and whose edges link two diagrams $w_1$ and $w_2$ if there exists some oriented hyperplane $J$ such that $w_2=w_1J$. 

\medskip \noindent 
A more precise description of the identification between $M(X,x)$ and $\widetilde{X}^{(1)}$ is the following:
\begin{table}[h!]
	\centering
		\begin{tabular}{l}
			\hspace{5cm} $M(X,x) \longleftrightarrow \left( \widetilde{X}, \widetilde{x} \right)$ \\ \\ 
			\begin{tabular}{c} $(x,\ast)$-diagram represented \\ by an $x$-legal word $w$ \end{tabular}  $\mapsto$ \begin{tabular}{c} path $\gamma \subset X$ \\ representing $w$ \end{tabular}  $\mapsto$  \begin{tabular}{c} lift $\widetilde{\gamma} \subset \widetilde{X}$ of $\gamma$ \\ starting from $\widetilde{x}$ \end{tabular} $\mapsto$ \begin{tabular}{c} ending \\ point of $\widetilde{\gamma}$ \end{tabular} \\ \\
			\begin{tabular}{c} $(x,\ast)$-diagram represented by the \\ $x$-legal word corresponding to $\gamma$ \end{tabular} $\mapsfrom$ \begin{tabular}{c} image \\ $\gamma \subset X$ of $\gamma$ \end{tabular} $\mapsfrom$ \begin{tabular}{c} path $\widetilde{\gamma} \subset \widetilde{X}$ \\ from $\widetilde{x}$ to $y$ \end{tabular} $\mapsfrom$ $y$
		\end{tabular}
\end{table}

\noindent
A diagram may be represented by several legal words. Such a word is \emph{reduced} if it has minimal length, i.e.\ it cannot be shortened by applying a sequence of cancellations, insertions and commutation. It is worth noticing that, in general, a diagram is not represented by a unique reduced legal word, but two such reduced words differ only by some commutations. (For instance, consider the homotopically trivial loop defined by the paths representing two of our reduced legal words, consider a disc diagram of minimal area bounded by this loop, and follow the proof of \cite[Theorem 4.6]{MR1347406}. Alternatively, use the embedding constructed in \cite[Section 4.1]{SpecialRH} (which does not use the present discussion) and conclude by applying the analogous statement which holds in right-angled Artin groups.) As a consequence, we can define the \emph{length} of a diagram as the length of any reduced legal word representing it. It is worth noticing that our length coincides with the length which is associated to the generating set given by the oriented hyperplanes in the groupoid $\mathcal{D}(X)$. The next lemma follows from this observation.

\begin{lemma}\label{lem:GeodesicsSpecial}
Let $D_1,D_2 \in M(X,x)$ be two $(x,\ast)$-diagrams. If $J_1 \cdots J_n$ is a reduced legal word representing $D_1^{-1}D_2$, then 
$$D_1, \ D_1J_1, \ D_1J_1J_2, \ldots, \ D_1J_1 \cdots J_n$$
is a geodesic from $D_1$ to $D_2$ in $M(X,x)$. Conversely, any geodesic between $D_1$ and $D_2$ arises in this way.
\end{lemma}

\noindent
Essentially by construction, one has:

\begin{fact}\label{fact:diagvsRAAG}
\emph{\cite[Fact 4.3]{SpecialRH}}
Let $X$ be a special cube complex and $x \in X$ a basepoint. Two $x$-legal words of oriented hyperplanes are equal in $\mathcal{D}(X,x)$ if and only if they are equal in the right-angled Artin group $A(\Delta X)$. 
\end{fact}

\noindent
As a consequence, a reduced legal word represents a trivial element in $\mathcal{D}(X,x)$ if and only if it is empty.

\subsection{Characterising cyclic hyperbolicity}

\noindent
Before turning to the proof of Theorem~\ref{thm:Special}, we record a couple of elementary general observations about cyclic hyperbolicity. 

\begin{lemma}\label{lem:CyclicVSabelian}
Let $G$ be a finitely generated group. If $G$ is weakly hyperbolic relative to a finite collection of finitely generated abelian subgroups, then it is cyclically hyperbolic.
\end{lemma}

\begin{proof}
Let $\mathcal{A}=\{A_1,\ldots, A_n\}$ be a collection of finitely generated abelian subgroups such that $G$ is weakly hyperbolic relative to $\mathcal{A}$. Let $\mathcal{Z}$ denote a collection of cyclic subgroups obtained from free bases of the $A_i$. Let $Y_1$ denote the cone-off of $G$ over the cosets of subgroups in $\mathcal{A}$ and let $Y_2$ denote the cone-off of $G$ over the cosets of subgroups in $\mathcal{Z}$. Here, we think of a cone-off as the graph obtained by adding an edge between any two vertices that belong to a common coset. Thus, the identity $G \to G$ induces a map $\varphi : Y_1 \to Y_2$. Observe that $\varphi^{-1}(Z)$ has diameter one for every coset $Z$ of a subgroup in $\mathcal{Z}$, and that $\varphi(A)$ has diameter at most $\max( \mathrm{rank}(A_i), 1 \leq i \leq n)$ for every coset $A$ of a subgroup in $\mathcal{A}$. This implies that $\varphi$ is a quasi-isometry. Consequently, $Y_2$ has to be hyperbolic, proving that $G$ is cyclically hyperbolic. 
\end{proof}

\begin{lemma}\label{lem:FiniteIndexHyp}
Let $G$ be a finitely generated group and $H \leq G$ a finite-index subgroup. Then $G$ is cyclically hyperbolic if and only if so is $H$.
\end{lemma}

\begin{proof}
Assume that $H$ is weakly hyperbolic relative to $\mathcal{H}=\{ \langle h_1 \rangle, \ldots, \langle h_n \rangle \}$ for some $h_1, \ldots, h_n \in H$. The inclusion $H \hookrightarrow G$ induces a quasi-isometry from the cone-off of $H$ over the cosets of subgroups in $\mathcal{H}$ to the cone-off of $G$ over the cosets of subgroups in $\mathcal{H}$, so the latter has to be hyperbolic, proving that $G$ is cyclically hyperbolic.

\medskip \noindent
Conversely, assume that $G$ is weakly hyperbolic relative to $\mathcal{G}=\{ \langle g_1 \rangle, \ldots, \langle g_n \rangle \}$ for some $g_1, \ldots, g_n \in G$. There exists some $k \geq 1$ such that $g_1^k,\ldots, g_n^k \in H$. Set $\mathcal{G}':=\{ \langle g_1^k \rangle, \ldots, \langle g_n^k \rangle\}$. Observe that the identity $G \to G$ induces a quasi-isometry from the cone-off of $G$ over the cosets of subgroups in $\mathcal{G}$ to the cone-off of $G$ over the cosets of subgroups in $\mathcal{G}'$, so the latter is also hyperbolic. But, as before, the inclusion $H \hookrightarrow G$ induces a quasi-isometry from the cone-off of $H$ over the cosets of subgroups in $\mathcal{G}'$ to the cone-off of $G$ over the cosets of subgroups in $\mathcal{G}'$, so the former has to be hyperbolic, proving that $H$ is cyclically hyperbolic.
\end{proof}

\begin{proof}[Proof of Theorem~\ref{thm:Special}.]
Let $G$ be a virtually cocompact special group. According to Lemma~\ref{lem:FiniteIndexHyp}, we can assume without loss of generality that $G$ is cocompact special. Let $X$ be a special cube complex with $G$ as its fundamental group and fix a vertex $x_0 \in X$. Assume first that $G$ does not contain $\mathbb{F}_2 \times \mathbb{F}_2$ as a subgroup. From now on, we identify $G$ with $\mathcal{D}(X,x_0)$.

\medskip \noindent
For every $x \in X$, we fix an $(x_0,x)$-legal word $w(x)$. If $\Lambda \subset \Delta X$, we denote by $w(x) \langle \Lambda \rangle w(x)^{-1}$ the subgroup given by the diagrams represented by $(x_0,x_0)$-legal words of the form $w(x) \cdot w \cdot w(x)^{-1}$ where $w$ is a word of hyperplanes in $\Lambda$. Also, we denote by $w(x) \langle \Lambda \rangle$ the (convex) subcomplex of $M(X,x_0)$ generated by the vertices represented by $x_0$-legal words of the form $w(x) \cdot w$ where $w$ is again a word of hyperplanes in $\Lambda$. Set
$$\mathcal{H} := \{ w(x) \langle \Lambda \rangle w(x)^{-1} \text{ abelian} \mid x \in X, \Lambda \subset \Delta X\}.$$ 
We claim that $G$ is weakly hyperbolic relative to $\mathcal{H}$.  This will imply, according to Lemma~\ref{lem:CyclicVSabelian}, that $G$ is cyclically hyperbolic, as desired.  Let $Y$ denote the cone-off of $M(X,x_0)$ over the $\mathcal{D}(X,x_0)$-translates of the subcomplexes in 
$$\{ w(x) \langle \Lambda \rangle \mid w(x) \langle \Lambda \rangle w(x)^{-1} \text{ abelian}, x \in X, \Lambda \subset \Delta X \}.$$ 
Our goal is to deduce from Theorem~\ref{thm:WhenHyp} that $Y$ is hyperbolic, which is sufficient for our purpose according to Lemma~\ref{lem:MilnorSvarc}.

\medskip \noindent
So let $R : [0,a] \times [0,b] \to M(X,x_0)$ be a flat rectangle with $a,b$ larger than the number of vertices in $X$. For convenience, we identify $[0,a] \times [0,b]$ with its image in $M(X,x_0)$. Let $w$ the $(x_0,\ast)$-diagram representing the vertex $(0,0)$. Let $\Phi \subset \Delta X$ denote the collection of the hyperplanes labelling the edges in $[0,a] \times \{0\}$ (or equivalently, in $[0,a] \times \{i\}$ for every $0 \leq i \leq b$) and $\Psi \subset \Delta X$ the hyperplanes labelling the edges in $\{0\} \times [0,b]$ (or equivalent, in $\{i\} \times [0,b]$ for every $0 \leq i \leq a$). Notice that every hyperplane in $\Phi$ is transverse to every hyperplane in $\Psi$, i.e.\ the subgraph in $\Delta X$ spanned by $\Phi \cup \Psi$ is a join $\Phi \ast \Psi$, and that $[0,a] \times [0,b] \subset w \langle \Phi \ast \Psi \rangle$. Also, notice that 
$$w \langle \Phi \ast \Psi \rangle w^{-1}= w \langle \Phi \rangle w^{-1} \times w \langle \Psi \rangle w^{-1}$$
where $ w \langle \Phi \rangle w^{-1}$ and $w \langle \Psi \rangle w^{-1}$ are both non-trivial since $a,b$ are sufficiently large. Indeed,  because $a$ is larger than the number of vertices in $X$ , $[0,a] \times \{0\}$ must contain at least two vertices $p,q$ with the same image in $X$.  Thinking of $p$ and $q$ as $(x_0,\ast)$-diagrams, they have the same terminus , so we get a diagram 
$$w \cdot w^{-1}q \cdot p^{-1}w \cdot w^{-1} \in w \cdot \langle \Phi \rangle \cdot w^{-1}$$
that is not trivial because Lemma~\ref{lem:GeodesicsSpecial} shows that $qp^{-1}$ is itself non-trivial. The same argument applies to $w \langle \Psi \rangle w^{-1}$. 

\medskip \noindent
Because $\mathcal{D}(X,x_0)$ does not contain $\mathbb{F}_2 \times \mathbb{F}_2$, one of our two factors does not contain $\mathbb{F}_2$, i.e.\ is abelian. Say that $w \langle \Phi \rangle w^{-1}$ is abelian. Then, for every $0 \leq i \leq b$, $[0,a] \times \{i\}$ lies in $u \langle \Phi \rangle$ for some word $u$ with $u \langle \Phi \rangle u^{-1}$ abelian. But, if $x$ denotes the terminus of a word representing $u$, then $u \langle \Phi \rangle$ is a $\mathcal{D}(X,x_0)$-translate of $w(x) \langle \Phi \rangle$, namely: $u \langle \Phi \rangle = uw(x)^{-1} \cdot w(x) \langle \Phi \rangle$. Consequently, $[0,a] \times \{i\}$ has diameter $2$ in $Y$. We conclude that Theorem~\ref{thm:WhenHyp} applies, proving that $\mathcal{D}(X,x_0)$ is weakly hyperbolic relative to $\mathcal{H}$, and finally that $\mathcal{D}(X,x_0)$, or equivalently $G$, is cyclically hyperbolic according to Lemma~\ref{lem:CyclicVSabelian}.

\medskip \noindent
Conversely, if $G$ contains $\mathbb{F}_2 \times \mathbb{F}_2$ as a subgroup, then we know from Corollary~\ref{cor:F2xF2} that $G$ is not cyclically hyperbolic.
\end{proof}

\section{A few examples}\label{section:Examples}

\noindent
In this final section, we discuss applications of our general results to a few classical families of groups. Unfortunately, it would be too long to describe the cubulations of all these groups, so several statements below are left without a proof. This is also justified by the fact that, even though they take place in different frameworks, the proofs are essentially all the same. 

\medskip \noindent
The key idea underlying this section is that, in the examples below, it is easier proving that a given group is cyclically hyperbolic rather than proving that it does not contain $\mathbb{F}_2 \times \mathbb{F}_2$ as a subgroup. Roughly speaking, in order to prove that a group cyclically hyperbolic, we consider all the ``obvious'' infinite cyclic subgroups and apply Proposition~\ref{prop:NotHyp} in order to deduce that our group is weakly hyperbolic relative them. In some sense, it suffices to show that the group has no ``obvious'' subgroup $\mathbb{F}_2 \times \mathbb{F}_2$ in order to deduce that it has no subgroup $\mathbb{F}_2 \times \mathbb{F}_2$ at all.

\paragraph{Right-angled Coxeter groups.} Given a simplicial graph $\Gamma$, the \emph{right-angled Coxeter group} $A(\Gamma)$ is defined by the presentation
$$\langle u \text{ vertex in } \Gamma \mid u^2=1 \ (u \text{ vertex}), \ [u,v]=1 \ (\{u,v\} \text{ edge in } \Gamma) \rangle.$$
The cube-completion $X(\Gamma)$ of the Cayley graph $\mathrm{Cayl}(C(\Gamma),V(\Gamma))$ is a CAT(0) cube complex on which the derived subgroup of $C(\Gamma)$ acts specially. In particular, $C(\Gamma)$ is virtually cocompact special. 

\begin{prop}\label{prop:RACGcyclhyp}
Let $\Gamma$ be a finite simplicial graph. The right-angled Coxeter group $C(\Gamma)$ is cyclically hyperbolic if and only if $\Gamma$ does not contain an induced copy of $K_{3,3}$, $K_{3,3}^+$, or $K_{3,3}^{++}$. 
\end{prop}

\noindent
Here, $K_{3,3}$ refers to the complete bipartite graph, which is also the join of two triples of isolated vertices; $K_{3,3}^+$ refers to the join of a triple of isolated vertices with $K_1 \sqcup K_2$  (i.e.\ the disjoint union of an isolated vertex with a single edge); and finally $K_{3,3}^{++}$ refers to the join of two copies of $K_1 \sqcup K_2$. See Figure~\ref{Graphs}. Our proof of the proposition relies on the following elementary observation:

\begin{lemma}\label{lem:RACGflat}
A simplicial graph $\Gamma$ decomposes as the join of a complete graph and pairs of isolated vertices if and only if it does not contain an induced subgraph that is either $K_3^{\mathrm{opp}}$ (i.e.\ a triple of isolated vertices) or $K_1 \sqcup K_2$ (i.e.\ the disjoint union of an isolated vertex with a single edge).
\end{lemma}

\begin{proof}
Clearly, a simplicial graph contains $K_3^{\mathrm{opp}}$ or $K_1 \sqcup K_2$ as an induced subgraph if and only if it has a vertex that this not adjacent to at least two other vertices. In other words, $\Gamma$ does not contain an induced copy of $K_3^{\mathrm{opp}}$ or $K_1 \sqcup K_2$ if and only if each of its vertices is adjacent to all the other vertices but possibly one, which amounts to saying that the opposite graph of $\Gamma$ is a disjoint union of isolated vertices and edges.
\end{proof}

\begin{proof}[Proof of Proposition~\ref{prop:RACGcyclhyp}.]
If $\Gamma$ contains an induced subgraph $\Lambda$ isomorphic to $K_{3,3}$ (resp.  $K_{3,3}^+$, $K_{3,3}^{++}$), then $\langle \Lambda \rangle$ is isomorphic to $(\mathbb{Z}_2 \ast \mathbb{Z}_2 \ast \mathbb{Z}_2)^2$ (resp. $(\mathbb{Z}_2 \ast \mathbb{Z}_2 \ast \mathbb{Z}_2) \times ( \mathbb{Z}_2 \ast \mathbb{Z}_2^2)$, $(\mathbb{Z}_2 \ast \mathbb{Z}_2^2)^2$), which contains a subgroup isomorphic to $\mathbb{F}_2 \times \mathbb{F}_2$. It follows from Corollary~\ref{cor:F2xF2} that $C(\Gamma)$ is not cyclically hyperbolic.

\medskip \noindent
Conversely, assume that $\Gamma$ does not contain an induced copy of $K_{3,3}$, $K_{3,3}^+$, or $K_{3,3}^{++}$. Let $\mathcal{H}$ denote the collection of all the subgroups in $C(\Gamma)$ of the form $\langle uv \rangle$ where $u,v \in V(\Gamma)$ are two non-adjacent vertices. We claim that $C(\Gamma)$ is weakly hyperbolic relative to $\mathcal{H}$. For this purpose, we denote by $Y$ the cone-off of $X(\Gamma)$ over the cosets of subgroups in $\mathcal{H}$ and we want to deduce from Theorem~\ref{thm:WhenHyp} that $Y$ is hyperbolic.

\medskip \noindent
Let $R : [0,a] \times [0,b] \hookrightarrow X(\Gamma)$ be a flat rectangle. For convenience, we identify $[0,a] \times [0,b]$ with its image in $X(\Gamma)$. Let $\Phi$ (resp. $\Psi$) denote the subgraph of $\Gamma$ generated by the vertices that label all the edges dual to hyperplanes crossing $[0,a] \times \{0\}$ (resp. $\{0\} \times [0,b]$). Because every hyperplane crossing $[0,a] \times \{0\}$ is transverse to every hyperplane crossing $\{0\} \times [0,b]$, every vertex in $\Phi$ is adjacent to every vertex in $\Psi$. Observe that, if $\Phi,\Psi$ both contain an induced copy of $K_3^\mathrm{opp}$ or $K_1 \sqcup K_2$, then $\Gamma$ contains an induced copy of $K_{3,3}$, $K_{3,3}^+$, or $K_{3,3}^{++}$. It follows from Lemma~\ref{lem:RACGflat} that one of $\Phi,\Psi$, say $\Phi$, is a join of a complete graph $\Lambda$ with union of pairs of isolated vertices $\{u_1,v_1\}, \ldots, \{u_n,v_n\}$. Because, for every $0 \leq i \leq b$, $[0,a] \times \{i\}$ is contained in a coset of $\langle \Phi \rangle$, it suffices to show that $\langle \Phi \rangle$ is bounded in $Y$ in order to deduce that Theorem~\ref{thm:WhenHyp} applies. But $\langle \Phi \rangle$ decomposes as the product of $\langle \Lambda \rangle$, which has diameter at most 
 $$\mathrm{clique}(\Gamma) := \text{maximal cardinality of a complete subgraph in } \Gamma$$  
in $X(\Gamma)$ and a fortiori in $Y$, and of the lines $\langle u_i,v_i \rangle$, which have diameter at most $2$ in $Y$ by our choice of $\mathcal{H}$. Consequently, $\langle \Phi \rangle$ has diameter at most $\mathrm{clique}(\Gamma) + |V(\Gamma)|$ in $Y$. 

\medskip \noindent
We conclude that $C(\Gamma)$ is weakly hyperbolic relative to $\mathcal{H}$, as claimed, and finally that $C(\Gamma)$ is cyclically hyperbolic.
\end{proof}

\begin{remark}
Observe that Proposition~\ref{prop:RACGcyclhyp} contradicts the main theorem from \cite{wrelhypRACG}, which implies that every right-angled Coxeter group is cyclically hyperbolic. However, this theorem claims in particular that 
$$G:=\langle a_1,a_2,a_3,b_1,b_2,b_3 \mid a_i^2=b_j^2= [a_i,b_j]=1 \ (1 \leq i,j \leq 3) \rangle,$$ 
which is a product of two subgroups $A:= \langle a_1,a_2,a_3 \rangle$ and $B:= \langle b_1,b_2,b_3 \rangle$ isomorphic to $\mathbb{Z}_2 \ast \mathbb{Z}_2 \ast \mathbb{Z}_2$, is weakly hyperbolic relative to $\{\langle a_1,a_2 \rangle, \langle a_2,a_3\rangle, \langle a_1,a_3\rangle\}$. But this is clearly false since the corresponding cone-off $Y$ of the Cayley graph of $G$ is quasi-isometric to the product between the cone-off $C$ of $\mathrm{Cayl}(A,\{a_1,a_2,a_3\})$ over the cosets of the subgroups in $\{\langle a_1,a_2\rangle, \langle a_2,a_3 \rangle, \langle a_1,a_3\rangle\}$ with $\mathrm{Cayl}(B,\{b_1,b_2,b_3\})$. The latter Cayley graph is of course unbounded, but $C$ is also unbounded (for instance, the infinite word $(a_1a_2a_3)^\infty$ labels an infinite geodesic ray in the Cayley graph of $A$ that remains unbounded in $C$). So the cone-off $Y$ cannot be hyperbolic, proving that $G$ is not weakly hyperbolic relative to $\{\langle a_1,a_2 \rangle, \langle a_2,a_3\rangle, \langle a_1,a_3\rangle\}$. 
\end{remark}

\paragraph{Right-angled Artin groups.} Given a simplicial graph $\Gamma$, the \emph{right-angled Artin group} $A(\Gamma)$ is defined by the presentation
$$\langle u \text{ vertex in } \Gamma \mid [u,v]=1 \ (\{u,v\} \text{ edge in } \Gamma) \rangle.$$
By reproducing almost word for word our argument for right-angled Coxeter groups, we find that:

\begin{prop}\label{prop:RAAGcyc}
Let $\Gamma$ be a finite simplicial graph. The right-angled Artin group $A(\Gamma)$ is cyclically hyperbolic if and only if $\Gamma$ does not contain an induced $4$-cycle.
\end{prop}

\noindent
Alternatively, we can also apply \cite{MR2475886} that proves that $A(\Gamma)$ contains $\mathbb{F}_2 \times \mathbb{F}_2$ if and only if $\Gamma$ contains an induced $4$-cycle. Since right-angled Artin groups are cocompact special, the desired conclusion follows immediately from Theorem~\ref{thm:Special}.

\paragraph{Graph products of finite groups.} Generalising right-angled Artin and Coxeter groups, the \emph{graph product} $\Gamma \mathcal{G}$ associated to a simplicial graph $\Gamma$ and a collection of groups $\mathcal{G}=\{G_u \mid u \in V(\Gamma)\}$ indexed by the vertices of $\Gamma$, referred to as \emph{vertex-groups}, is defined as the quotient
$$\left( \underset{u \in V(\Gamma)}{\ast} G_u \right) / \langle \langle [g,h]=1 \ (g \in G_u, h\in G_v, \{u,v\} \text{ edge in } \Gamma) \rangle \rangle.$$
Mimicking the arguments used in the proof of Proposition~\ref{prop:RACGcyclhyp}, one can prove that:

\begin{prop}\label{prop:GraphProducts}
Let $\Gamma$ be a finite simplicial graph and $\mathcal{G}$ a collection of non-trivial finite groups indexed by $V(\Gamma)$. The graph product $\Gamma \mathcal{G}$ is cyclically hyperbolic if and only if it $\Gamma$ does not contain induced subgraphs as given by Figure~\ref{Graphs}, i.e.\
\begin{itemize}
	\item a copy of $K_{3,3}$, $K_{3,3}^+$, or $K_{3,3}^{++}$;
	\item a copy of $K_{2,2}$ with two adjacent vertices indexing vertex-groups of sizes $\geq 3$;
	\item a copy of $K_{3,3}$ with one vertex of degree $3$ indexing a vertex-group of size $\geq 3$;
	\item a join of $K_1 \sqcup K_2$ with two isolated vertices, one of them indexing a vertex-group of size $\geq 3$.
\end{itemize}
\end{prop}
\begin{figure}
\begin{center}
\includegraphics[width=0.7\linewidth]{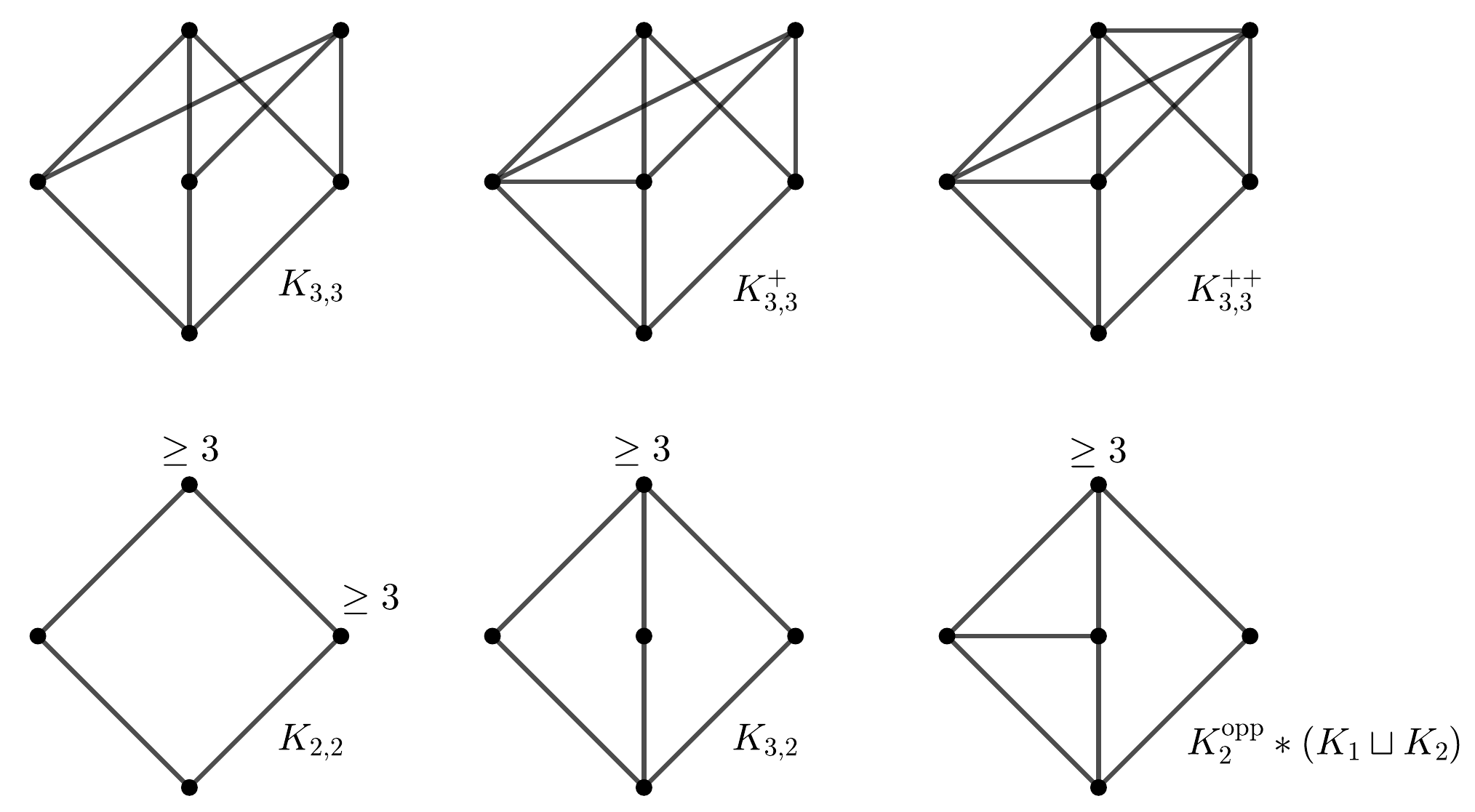}
\caption{Configurations that prevent being cyclically hyperbolic.}
\label{Graphs}
\end{center}
\end{figure}

\noindent
Even though graph products of finite groups act geometrically on CAT(0) cube complexes and are virtually cocompact special (see for instance \cite{MR2413337, MR3010817}), the best way to reproduce the arguments we used for right-angled Coxeter groups is to make our graph products of finite groups act on \emph{quasi-median graphs} as explained in \cite[Section~8]{Qm} (see also \cite[Section~2.2]{MR4295519} for a more direct description). Because quasi-median graphs share essentially the same geometry as CAT(0) cube complexes (see \cite[Section~2]{Qm}), there is no difficulty here. Nevertheless, we do not repeat the arguments and leave Proposition~\ref{prop:GraphProducts} without a proof.

\medskip \noindent
One easily sees that, if the simplicial graph $\Gamma$ from Proposition~\ref{prop:GraphProducts} contains one of the configurations given by Figure~\ref{Graphs}, then the graph product $\Gamma \mathcal{G}$ contains a subgroup isomorphic to $\mathbb{F}_2 \times \mathbb{F}_2$. Consequently, the combination of Proposition~\ref{prop:GraphProducts} and Corollary~\ref{cor:F2xF2} yields the following generalisation of Corollary~\ref{cor:ProductInRACG}:

\begin{cor}
Let $\Gamma$ be a simplicial graph and $\mathcal{G}$ a collection of non-trivial finite groups indexed by $V(\Gamma)$. The graph product $\Gamma \mathcal{G}$ contains $\mathbb{F}_2 \times \mathbb{F}_2$ as a subgroup if and only if $\Gamma$ contains one of the configurations given by Figure~\ref{Graphs}.
\end{cor}

\paragraph{Graph braid groups.} Given a topological space $X$ and an integer $n \geq 1$, the \emph{configuration space of $n$ points in $X$} is 
$$C_n(X)= \{ (x_1, \ldots, x_n) \in X^n \mid \text{for every $i \neq j$, $x_i \neq x_j$} \}.$$
So a point in $C_n(X)$ is the data of $n$ ordered and pairwise distinct points in $X$. The corresponding configuration space of unordered collections of points is the quotient $UC_n(X)= C_n(X) / \mathfrak{S}_n$ where the symmetric group $\mathfrak{S}_n$ acts by permuting the coordinates. Given an initial configuration $\ast \in UC_n(X)$, the \emph{braid group} $B_n(X,\ast)$ is the fundamental group of $UC_n(X)$ based at $\ast$. Basically, it is the group of trajectories of $n$ points  or \emph{particles}  in $X$, starting and ending at $\ast$ (not necessarily in the same order), up to isotopy. When $X$ is a one-dimensional CW-complex, one refers to a \emph{graph braid group}. In this case, as observed in \cite{GraphBraidGroups}, $UC_n(X)$ can be discretised in order to obtained a (homotopy equivalent) nonpositively curved cube complex, which turns out to be special as shown in \cite{CrispWiest}. 

\medskip \noindent
In \cite{MR4227231}, we exploited the formalism introduced in \cite{SpecialRH} that allows us to study graph braid groups like right-angled Artin groups. By adapting the proof of Proposition~\ref{prop:RACGcyclhyp} to this framework, one obtains the following statement:

\begin{prop}\label{prop:Braid}
Let $\Gamma$ be a connected one-dimensional CW-complex and $n \geq 0$ an integer. The graph braid group $B_n(\Gamma)$ is cyclically hyperbolic if and only if, for all disjoint connected subcomplexes $\Lambda_1,\Lambda_2 \subset \Gamma$ and for all integers $n_1,n_2 \geq 0$ satisfying $n_1+n_2 \leq n$, there exists $i \in \{1,2\}$ such that $B_{n_i}(\Lambda_i)$ is cyclic.
\end{prop}

\noindent
It is worth noticing that \cite[Lemmas~4.3 and~4.16]{MR4227231} characterise exactly when a graph braid group is cyclic, so Proposition~\ref{prop:Braid} can be formulated just in terms of forbidden subcomplexes. For instance, for graph braid groups with two particles:

\begin{cor}
Let $\Gamma$ be a connected one-dimensional CW-complex. The graph braid group $B_2(\Gamma)$ is cyclically hyperbolic if and only if $\Gamma$ does not contains two disjoint connected subcomplexes that both contain at least two cycles.
\end{cor}

\noindent
It is worth noticing that, if $\Gamma$ contains two vertices of degree $\geq 3$ and if $n \geq 6$, then $B_n(\Gamma)$ contains a subgroup isomorphic to $B_3(Y) \times B_3(Y)$, where $Y$ denotes a star with three arms, which is a product of two non-abelian groups. Therefore, if $B_n(\Gamma)$ is cyclically hyperbolic with $n \geq 6$, then $\Gamma$ has to contain at most one vertex of degree $\geq 3$. In the terminology of \cite{MR4227231}, this means that $\Gamma$ is a \emph{flower graph}, and, according to \cite[Corollary~4.7]{MR4227231}, this implies that $B_n(\Gamma)$ is free. Thus, there do not exist interesting examples of cyclically hyperbolic graph braid groups with a high number of particles. 

\medskip \noindent
As a concrete application of Proposition~\ref{prop:Braid}, the graph braid group $B_n(K_m)$ with $n \geq 2$ particles over the complete graph $K_m$ is cyclically hyperbolic if and only if $m \leq 7$. Observe that one obtains infinitely many examples that are not free, since $B_n(K_m)$ contains $B_2(K_5)$ as a subgroup, which is a surface group \cite[Example~5.1]{GraphBraidGroups}, for all $n \geq 2$ and $m \geq 6$. This contrasts with hyperbolic and toric relatively hyperbolic graph braid groups \cite[Examples~4.10 and~4.27]{MR4227231}.

\paragraph{BMW-groups.} It follows from Proposition~\ref{prop:NotHyp} that groups acting on products of trees (with infinite boundaries) are not cyclically hyperbolic. This includes the so-called BMW-groups \cite{MR3931408}. In \cite{MR2174099}, it is proved that there exists a BMW-group $G$ that is commutative-transitive (i.e.\ for all $a,b,c \in G$, if $[a,b]=[b,c]=1$ then $[a,c]=1$). Of course, $G$ cannot contain a subgroup isomorphic to $\mathbb{F}_2 \times \mathbb{F}_2$ (and even to $\mathbb{F}_2 \times \mathbb{Z}$), proving that the converse of Corollary~\ref{cor:F2xF2} provided by Theorem~\ref{thm:Special} to virtually cocompact special groups cannot be generalised to arbitrary groups acting geometrically on CAT(0) cube complexes.

\addcontentsline{toc}{section}{References}

\bibliographystyle{alpha}
{\footnotesize\bibliography{CyclicHypCCI}}

\Address

%\addcontentsline{toc}{section}{Index}
%
%\printindex

\end{document}